\documentclass[12pt]{amsart}
\usepackage[T1]{fontenc}

\date{\today}

\usepackage{amsmath,amssymb,ifthen}
\usepackage{fullpage}
\usepackage{graphicx,psfrag,subfigure}
\usepackage{color}
\usepackage{moreverb}

\def\lagside{\ell_\mathrm{lag}}
\def\morside{\ell_\mathrm{mor}}

\def\R{{\mathbb R}}

\def\Pmat{\mathbf{P}}
\def\Amat{\mathbf{A}}
\def\Bmat{\mathbf{B}}
\def\TransferMat{\mathbf{T}}
\def\DiagMat{\mathbf{D}}

\def\sysmat{\mathbf{C}}

\def\supp{\mathrm{supp}}

\def\xx{\mathbf{x}}
\def\yy{\mathbf{y}}

\def\HH{\boldsymbol{H}} % H^d . . . d-dim Sobolev spaces
\def\blf{\widehat{a}}
\def\blfalt{a}

\def\KK{\mathcal{K}}
\def\TT{\mathcal{T}}

\def\spec{\mathrm{spec}}
\def\dim{\mathrm{dim}}
\def\dist{\mathrm{dist}}
\def\diam{\mathrm{diam}}
\def\norm#1#2{\|#1\|_{#2}}
\def\seminorm#1#2{\vert #1\vert_{#2}}
\def\eps{\varepsilon}
\def\dual#1#2{\langle#1\,,\,#2\rangle}

\def\normal{\boldsymbol{n}}
\def\curl{{\mathbf{curl}}}

\def\slp{V} % simple-layer potential
\def\hyp{W} % hypersingular integral operator

\def\hmax{h}
\def\hmin{\underline{h}}

\def\evmax{\lambda_{\mathrm{max}}}
\def\evmin{\lambda_{\mathrm{min}}}

\def\Evmax{\Lambda_{\mathrm{max}}}
\def\Evmin{\Lambda_{\mathrm{min}}}

\def\bignull{\boldsymbol{0}} 
\def\pphi{\boldsymbol{\Phi}}

\def\Mmat{\mathbf{M}}
\def\rr{\mathbf{r}}
\def\ff{\mathbf{f}}

\newtheorem{theorem}{Theorem}
\newtheorem{proposition}[theorem]{Proposition}
\newtheorem{lemma}[theorem]{Lemma}

\newtheorem{assumption}{Assumption}

\newtheorem{remark}[theorem]{Remark}

\def\subsection#1
{
 \bigskip

 \refstepcounter{subsection}
 {\noindent\bf\arabic{section}.\arabic{subsection}.~#1.~}
}

\begin{document}%%%%%%%%%%%%%%%%%%%%%%%%%%%%%%%%%%%%%%%%%%%%%%%%%%%%%

%%%%%%%%%%%%%%%%%%%%%%%%%%%%%%%%%%%%%%%%%%%%%%%%%%%%%%%%%%%%%%%%%%%%%
% Title and Authors
\title[Wirebasket preconditioner for mortar BEM]
{A wirebasket preconditioner for the\\ mortar boundary element method}

\date{\today}

\author{Thomas~F\"uhrer}
\author{Norbert~Heuer}
\address{Facultad de Matem\'aticas, 
        Pontificia Universidad Cat\'olica de Chile,
        Avenida Vicu\~{n}a Mackenna 4860,
      Santiago, Chile}
\email{\{tofuhrer,nheuer\}@mat.puc.cl}

%%%%%%%%%%%%%%%%%%%%%%%%%%%%%%%%%%%%%%%%%%%%%%%%%%%%%%%%%%%%%%%%%%%%%
% Classification

%%%%%%%%%%%%%%%%%%%%%%%%%%%%%%%%%%%%%%%%%%%%%%%%%%%%%%%%%%%%%%%%%%%%%
% Acknowledgment
\thanks{{\bf Acknowledgments:}
  The first author is supported by CONICYT through FONDECYT project 3150012,
  and the second author by CONICYT through projects
  FONDECYT 1150056 and Anillo ACT1118 (ANANUM)}
%%%%%%%%%%%%%%%%%%%%%%%%%%%%%%%%%%%%%%%%%%%%%%%%%%%%%%%%%%%%%%%%%%%%%
% ABSTRACT
\begin{abstract}
We present and analyze a preconditioner of the additive Schwarz type for the
mortar boundary element method. As a basic splitting,
on each subdomain we separate the degrees of freedom
related to its boundary from the inner degrees of freedom.
The corresponding \emph{wirebasket-type} space decomposition is stable up
to logarithmic terms.
For the blocks that correspond to the inner degrees of freedom
standard preconditioners for the hypersingular integral operator on open boundaries
can be used. For the boundary and interface parts as well as the
Lagrangian multiplier space, simple diagonal preconditioners are optimal.
Our technique applies to quasi-uniform and non-uniform meshes of shape-regular
elements.
Numerical experiments on triangular and quadrilateral meshes confirm theoretical
bounds for condition and MINRES iteration numbers.
\end{abstract}
%%%%%%%%%%%%%%%%%%%%%%%%%%%%%%%%%%%%%%%%%%%%%%%%%%%%%%%%%%%%%%%%%%%%%
% Make Title
\maketitle

%%%%%%%%%%%%%%%%%%%%%%%%%%%%%%%%%%%%%%%%%%%%%%%%%%%%%%%%%%%%%%%%%%%%%
% CONTENTS
%%%%%%%%%%%%%%%%%%%%%%%%%%%%%%%%%%%%%%%%%
% INTRODUCTION
%%%%%%%%%%%%%%%%%%%%%%%%%%%%%%%%%%%%%%%%%
% INTRODUCTION
%%%%%%%%%%%%%%%%%%%%%%%%%%%%%%%%%%%%%%%%%
\section{Introduction}

In recent years, different variants of the non-conforming boundary element method (BEM) have been
developed. The underlying boundary integral equation is of the first kind with hypersingular operator.
\emph{Non-conformity} refers to the presence of discontinuous basis functions. (Note that, in the case
of integral equations of the second kind or first kind equations with weakly-singular operator,
conforming basis functions can be discontinuous.)
The first paper on non-conforming BEM considers a Lagrangian multiplier to deal with the homogeneous
boundary condition on open surfaces \cite{ghh09}. This technique was extended in
\cite{mortarBEM} to domain decomposition approximations, and is usually referred to
as \emph{mortar method}. In this paper we study preconditioners for the mortar BEM presented in
\cite{mortarBEM}. These are the first results on preconditioning techniques for linear systems
stemming from non-conforming boundary elements.

Our preconditioner is based on a decomposition of the approximation space and choosing
locally equivalent bilinear forms. It therefore fits the additive Schwarz framework.
There is a large amount of literature on the additive Schwarz method, mainly aiming
at finite element systems, see, e.g., \cite{QuarteroniV_99_DDM,SmithBG_96_DD,ToselliW_05_DDM}
for overviews. For additive Schwarz techniques applied to boundary elements dealing with
hypersingular operators see, e.g., \cite{HeuerS_99_ISH,StephanT_98_DDA,transtep96},
cf. also \cite{Stephan_00_MMh} for an overview. In particular, \cite{HeuerS_99_ISH} considers
a wirebasket-oriented splitting.
Graded meshes on curves, locally refined and anisotropic meshes (on surfaces) have been analyzed,
respectively, in \cite{ffps,HeuerS_03_ODD,Maischak_09_MAS}.
Other variants, also for hypersingular operators, consider overlapping decompositions and
multiplicative applications, see, e.g., \cite{Tran_00_OAS,MaischakST_00_MSA}.

In this paper we extend the additive Schwarz technique to mortar boundary elements. In this case,
due to the presence of a Lagrangian multiplier, system matrices have a saddle point structure.
This structure can be handled by using standard arguments aiming at the minimum residual method (MINRES).
More precisely, the spectrum of the system matrix (with or without preconditioner) is being
controlled by the spectrum of the main block and the singular values of the off-diagonal block
(arising due to the presence of the Lagrangian multiplier), see~\cite{wfs95}.
A second complication due to the non-conformity of the method is that the bilinear form $a(\cdot,\cdot)$
representing the hypersingular operator is replaced by the weakly
singular operator acting on surface differential operators (surface curl).
There are no standard preconditioners for this bilinear form. Our strategy is to split the
subspace of discontinuous basis functions $X_h^1$ from the rest of the approximation space $X_h$.
The remainder $X_h^0$ forms a subspace of the energy space of the hypersingular operator.
It turns out that the bilinear form $a(\cdot,\cdot)$ reduces to the standard one of the
hypersingular operator when restricted to $X_h^0$. In this way, standard preconditioners
like the ones mentioned previously can be applied to this block (actually, there are individual
blocks associated with each subdomain). Now, the other subspace $X_h^1$ contains all the
basis functions associated to interface or boundary nodes. In domain decomposition terms, it
is a wirebasket space and is related with the skeleton (or wirebasket) of a coarse mesh
which is formed by the subdomains of the underlying decomposition. In our case, basis functions
associated to the boundary of a subdomain $\Gamma_i$ can be decoupled from the other elements of $X_h^1$,
they form a subspace $X_{i,1}$. It turns out that the bilinear form $a(\cdot,\cdot)$ restricted
to $X_{i,1}$ is spectrally equivalent to a diagonal matrix
(a mass matrix related to the boundary of $\Gamma_i$). In this way, simple diagonal matrices can be
used (for the preconditioner) to reduce the problem of preconditioning the bilinear form
$a(\cdot,\cdot):\; X_h\times X_h\to\R$ to the standard one of hypersingular operators on each subdomain.
In our numerical examples we will use multilevel diagonal scaling from \cite{ffps} for these parts.

A priori error analysis for domain-oriented non-conforming boundary elements yields quasi-optimal
error estimates which are perturbed by (poly-) logarithmic terms depending on the mesh size,
see~\cite{ChoulyH_12_NDD,ghh09,mortarBEM}. These perturbations appear due to the non-existence
of a well-defined trace operator in the energy space of hypersingular operators.
It is unknown whether estimates of these perturbations are sharp. Naturally, such logarithmic
perturbations also appear in the analysis of additive Schwarz preconditioners, at least when considering
non-overlapping decompositions. Note that, in our method, we subtract a wirebasket space and this
amounts to trace operations at the boundaries of subdomains. Also in the case of finite elements,
such splitting operations cause logarithmic perturbations, see, e.g., \cite{dsw94}.
Eventually, our main result considers combinations of simple diagonal and
multilevel diagonal preconditioners and proves that they are optimal up to poly-logarithmic terms.
In some cases, logarithmic perturbations of condition number bounds
can be optimized by multiplying terms of the preconditioner by different logarithmic
weights, see again, e.g.,~\cite{dsw94}. In this paper we consider three different weightings
(\emph{Cases 1,2,3}) where \emph{Cases 2,3} are optimized to show a bound $O(|\!\log(\underline{h})|^4)$
for the condition number of the preconditioned system. Here, $\underline{h}$ denotes the minimum
of the diameters of all elements. In contrast, for \emph{Case 1} (which does not use logarithmic weights
for the parts dealing with the bilinear form $a(\cdot,\cdot)$) the theoretical bound is
$O(|\!\log(\underline{h})|^5)$, worse than the bounds for \emph{Cases 2,3}. Our numerical experiments,
on the other hand, indicate that \emph{Case 1} is superior to \emph{Cases 2,3}. This suggests that
some of the theoretical bounds used for the analysis are not sharp, at least in the particular situation
of our numerical examples.

An outline of the remainder of this paper is as follows. In the next section we present the model
problem, recall the definition of some Sobolev norms, and present a mortar discretization for
the model problem. In Section~\ref{sec:precond} we recall some results on the MINRES method,
present our subspace decompositions and corresponding preconditioners, and state the main
result (Theorem~\ref{thm:cond}). Proofs are given in Section~\ref{sec:main:proof}.
Some numerical experiments are reported in Section~\ref{sec:examples}.
They all confirm our theoretical estimates from Theorem~\ref{thm:cond}, though exhibit
smaller logarithmic perturbations than predicted. In particular, we also study the case
of locally refined meshes driven by adaptivity, where preconditioners behave as expected.
Let us note that we do not know of any a posteriori error analysis for mortar boundary elements.
The only known results concerning non-conforming BEM consider two-level (or $h-h/2$) estimators
applied to the Nitsche coupling \cite{DominguezH_14_PEA}, not the mortar coupling.

\subsubsection*{Notation}
We abbreviate estimates of the form $A\leq C\cdot B$ with some constant $C>0$ by $A\lesssim B$. 
In particular, we use this notation if $C$ is independent of the mesh size and the number of elements.
Analogously, we use $A\gtrsim B$ for $A\geq C \cdot B$. If both $A\lesssim B$ and $A\gtrsim B$ hold true, we use the
notation $A\simeq B$.
Moreover, $|x|$ denotes the Euclidean norm for a point $x\in\R^3$.

%%%%%%%%%%%%%%%%%%%%%%%%%%%%%%%%%%%%%%%%%
% MortarBEM
%%%%%%%%%%%%%%%%%%%%%%%%%%%%%%%%%%%%%%%%%
% MORTAR BEM
%%%%%%%%%%%%%%%%%%%%%%%%%%%%%%%%%%%%%%%%%
\section{Mortar boundary elements}\label{sec:mortar}
In this section we briefly recall some results on the mortar boundary element method.

\subsection{Model problem and functional analytic setting}
Let $\Gamma \subset \R^2 \times\{0\}$ denote a plane open surface with polygonal boundary $\partial\Gamma$. For
simplicity we refer to $\Gamma$ as a domain in $\R^2$.

We recall some definitions of Sobolev spaces. Let $S\subset \R^2$ be a bounded subset and define for $0<s<1$ the seminorm
\begin{align*}
  \seminorm{u}{H^{s}(S)}^2 := \int_S \int_S \frac{|u(x)-u(y)|^2}{|x-y|^{2(s+1)}} \,dx\,dy.
\end{align*}
Then, $H^s(S)$ is equipped with the norm
\begin{align*}
  \norm{u}{H^s(S)}^2 := \norm{u}{L^2(S)}^2 + \seminorm{u}{H^s(S)}^2,
\end{align*}
and $\widetilde H^s(S)$ is defined as the completion of $C_0^\infty(S)$ with respect to the norm
\begin{align*}
  \norm{u}{\widetilde H^s(S)}^2 := \seminorm{u}{H^s(S)}^2 + \int_S \frac{|u(x)|^2}{\dist(x,\partial S)^{2s}} \,dx,
\end{align*}
where $\dist(x,\partial S) := \inf\limits_{y\in \partial S} |x-y|$.
The dual spaces of $H^s(S)$, resp. $\widetilde H^s(S)$, are denoted by $\widetilde H^{-s}(S)$, resp. $H^{-s}(S)$.
Additionally, $\dual\cdot\cdot_S$ denotes the $L^2(S)$ scalar product, which is continuously extended to the duality
pairing on $\widetilde H^{-s}(S) \times H^s(S)$, resp. $H^{-s}(S) \times \widetilde H^s(S)$.

Let $\normal \in \R^3$ denote a normal vector on $\Gamma$, e.g., $\normal = (0,0,1)^T$. Define the hypersingular integral
operator (formally) by
\begin{align*}
  \hyp u(x) := -\frac{\partial}{\partial\normal_x} \int_\Gamma  u(y) \frac{\partial}{\partial\normal_y} \frac1{|y-x|}\,dy.
\end{align*}
It is well known that this operator extends to a continuous mapping between $\widetilde H^{1/2}(\Gamma)$ and
$H^{-1/2}(\Gamma)$.

Our model problem reads as follows: Given $f\in L^2(\Gamma)$ we seek for a solution $u\in\widetilde{H}^{1/2}(\Gamma)$
such that
\begin{align}\label{eq:conform:cont}
  \dual{\hyp u}{v}_\Gamma = \dual{f}v_\Gamma \quad\forall v\in \widetilde H^{1/2}(\Gamma).
\end{align}
The usual conforming boundary element method consists in replacing $\widetilde{H}^{1/2}(\Gamma)$ by a finite-dimensional subspace
$\widetilde X_h \subset \widetilde{H}^{1/2}(\Gamma)$ and seeking for a solution $\widetilde u_h \in \widetilde X_h$
such that
\begin{align*} %\label{eq:conform:discrete}
  \dual{\hyp \widetilde u_h}{\widetilde v_h}_\Gamma = \dual{f}{\widetilde v_h}_\Gamma \quad\forall \widetilde v_h\in \widetilde X_h.
\end{align*}
In this paper, we study preconditioners for a non-conforming scheme that is based on a decomposition of the surface $\Gamma$.
In the next section we introduce the corresponding subspace decomposition. The non-conforming method based on this
decomposition is called mortar method and is presented in Section~\ref{sec_mortar}.

\subsection{Subspace decomposition and meshes}
Let $\Gamma_1,\dots,\Gamma_N$ denote a decomposition into non-intersecting (open) polygonal subdomains giving rise
to the coarse mesh
\begin{align*}
  \TT := \{\Gamma_1,\dots,\Gamma_N\}\quad
  \text{with}\quad \overline\Gamma=\bigcup_{j=1}^N \overline\Gamma_j.
\end{align*}
Each subdomain $\Gamma_i$ is equipped with (a sequence of) regular and quasi-uniform meshes $\TT_i$.
The minimum and maximum diameters of elements of the meshes $\TT_i$ are
denoted by $\hmin_i$ and $\hmax_i$, respectively.
As in~\cite{mortarBEM} we assume without loss of generality that $\hmax_i<1$ and set
\begin{align*}
  \hmin := \min_{i=1,\dots,N} \hmin_i \quad\text{and}\quad \hmax:= \max_{i=1,\dots,N} \hmax_i.
\end{align*}
Let $\KK_i$ denote the set of nodes of $\TT_i$. We will also need the set of interior nodes
$\KK_i^0$ and the set of nodes on the boundary of $\Gamma_i$, $\KK_i^1:=\KK_i\setminus\KK_i^0$.
For a node $z_j \in\KK_i$ we denote by 
$\eta_j^{(i)}$ the (bi)linear basis function which satisfies $\eta_j^{(i)}(z_k) = \delta_{jk}$ for
all $z_k \in\KK_i$.
Introducing the space of piecewise (bi)linear functions
\begin{align*}
  X_{h,i} := \{ v \in C^0(\Gamma_i) \,:\, v = \sum_{z_j \in \KK_i} \alpha_j \eta_j^{(i)} \text{ with } \alpha_j\in\R \}
\end{align*}
we define the product spaces
\begin{align*}
  X_h := \prod_{i=1}^N X_{h,i} \subset H^{1/2}(\TT) := \prod_{i=1}^N H^{1/2}(\Gamma_i).
\end{align*}
We denote the respective degrees of freedom by $K_i := \#\KK_i =\dim(X_{h,i})$
and $K:= \sum_{i=1}^N K_i = \dim(X_h)$.
Note that by definition of $\KK_i$, elements of $X_h$
do not necessarily satisfy the homogeneous boundary condition on $\partial\Gamma$ nor continuity across interfaces
$\partial\Gamma_i\cap\partial\Gamma_j$.

For the mortar BEM, we denote the interface of two
neighboring subdomains $\Gamma_i \neq \Gamma_j$ by
\(
  \gamma_{ij} := \mathrm{int}(\overline\Gamma_i\cap \overline\Gamma_j)
\)
where ``int'' refers to the (relative) interior.
Also, let $\diam(S)$ denote the diameter of a set $S\subseteq\R^3$.
We will need the following assumption.

\begin{assumption}\label{ass:interface}
  Each non-empty interface $\gamma_{ij}$ ($i,j=1,\dots,N$, $i\neq j$) consists of an entire edge of $\Gamma_i$ or $\Gamma_j$.
  If $\diam(\partial\Gamma\cap\partial\Gamma_i)>0$, then $\partial\Gamma\cap\partial\Gamma_i$ is a union of edges of
  $\Gamma_i$.
\end{assumption}

\begin{figure}
  \begin{center}
    \includegraphics[width=0.35\textwidth]{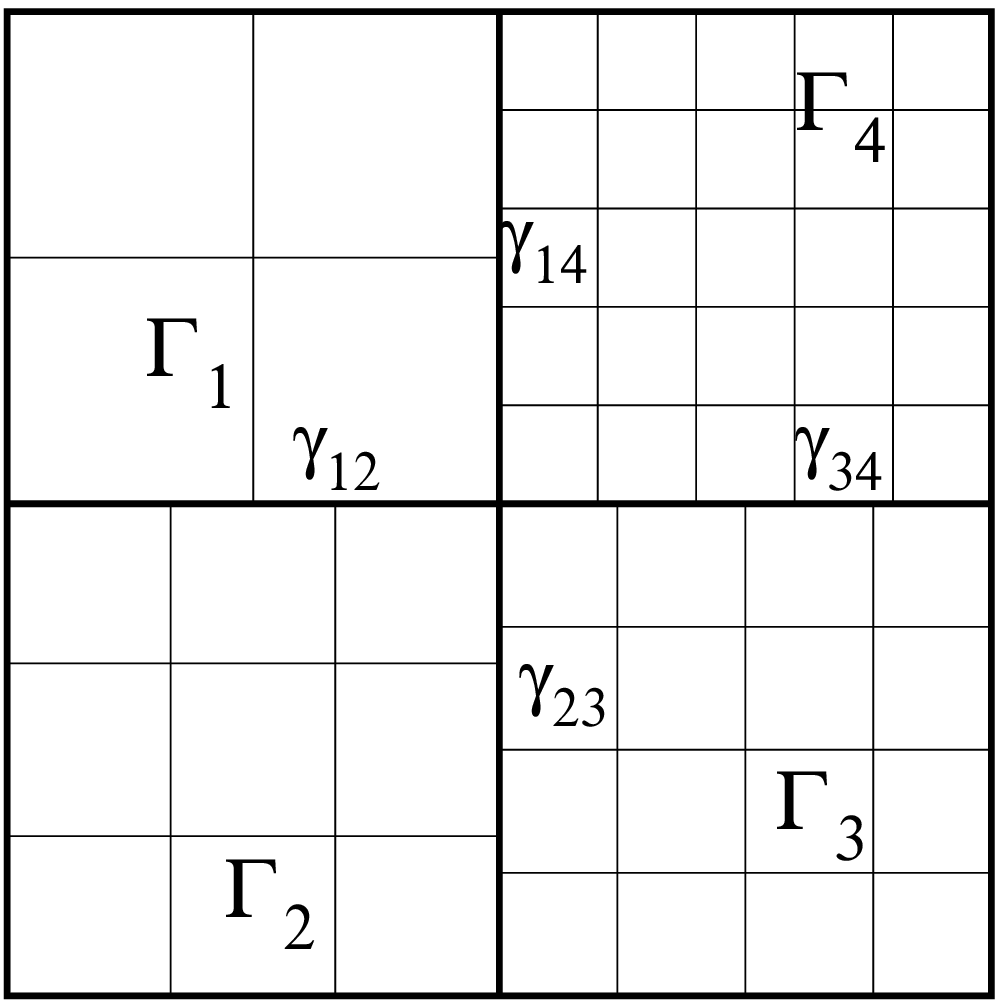}
    \hspace*{4ex}
    \includegraphics[width=0.35\textwidth]{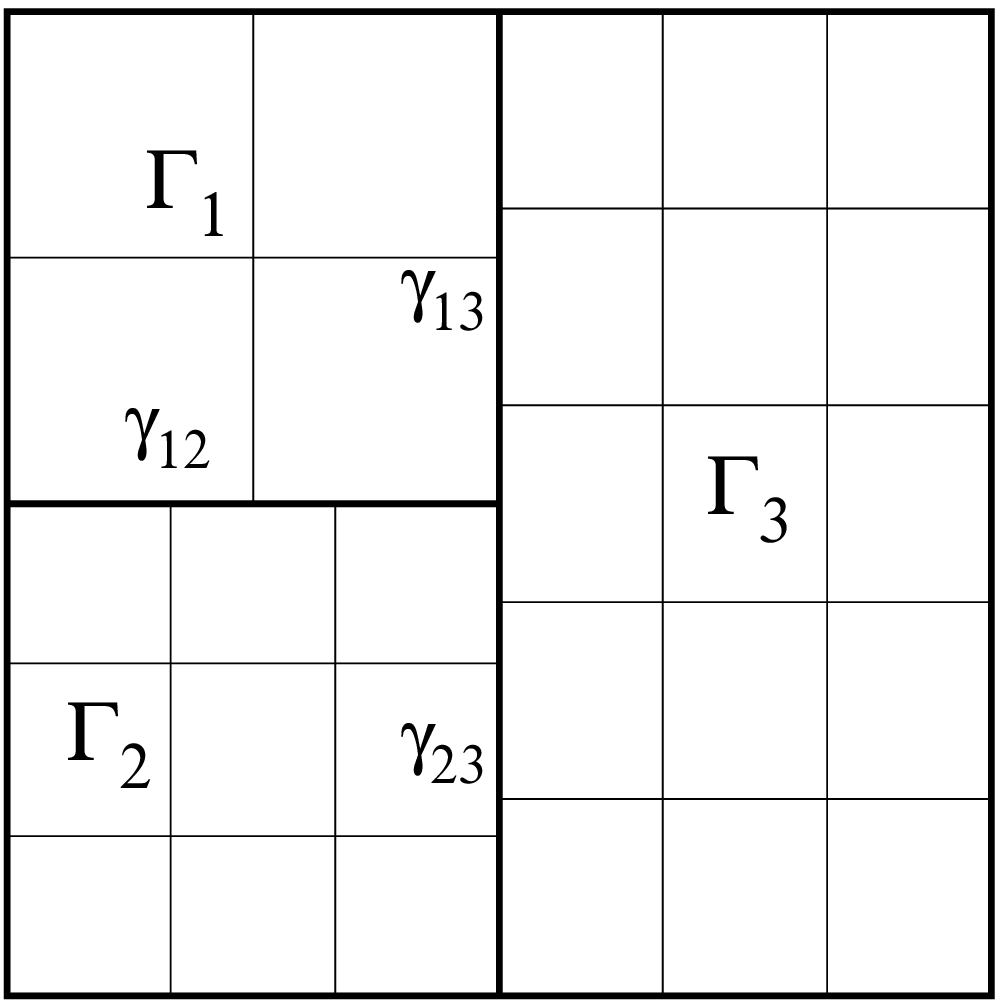}
  \end{center}
  \caption{Subspace decompositions with non-conforming meshes.}
  \label{fig:mortarDecomp}
\end{figure}

Figure~\ref{fig:mortarDecomp} shows examples of two different subspace decompositions with non-conforming meshes.
Given the skeleton
\begin{align*}
  \gamma:= \bigcup_{i=1}^N \partial\Gamma_i,
\end{align*}
we infer from Assumption~\ref{ass:interface} that $\gamma$ is covered by a set of non-intersecting interface edges
and boundary edges
\begin{align*}
  \tau := \{\gamma_1,\dots,\gamma_L\} \quad\text{with}\quad \overline\gamma = \bigcup_{\ell=1}^L \overline\gamma_\ell.
\end{align*}
For each interface edge $\gamma_\ell$, let $\lagside \in\{1,\dots,N\}$ resp. $\morside \in\{1,\dots,N\}$ denote the indices of the subdomains $\Gamma_{\lagside}$ resp.
$\Gamma_{\morside}$ such that
\begin{align*}
  \gamma_\ell = \gamma_{\lagside,\morside} \quad\text{and}\quad \gamma_\ell \text{ is an edge of } \Gamma_{\lagside}.
\end{align*}
In particular, we set $\lagside:=i$ if $\gamma_\ell$ is a boundary edge, i.e., $\gamma_\ell \subseteq \partial\Gamma_i \cap
\partial\Gamma$ so that we can handle the homogeneous boundary conditions and the interface conditions simultaneously.

On each $\gamma_\ell$ we introduce a mesh $\tau_\ell$ such that the following assumption is satisfied.

\begin{assumption}\label{ass:interfacemesh}
  The mesh $\tau_\ell$ of $\gamma_\ell$ is a strict coarsening of $\TT_{\lagside}|_{\gamma_\ell}$.
  In particular, any element of $\tau_\ell$ covers at least two elements of $\TT_{\lagside}|_{\gamma_\ell}$.

  Moreover, the diameters of the elements in $\tau_\ell$ and $\TT_{\lagside}|_{\gamma_\ell}$ are comparable, i.e., there
  exists a constant $C>0$ such that
  \begin{align*}
    C^{-1} \, \diam(t) \leq \diam(T) \leq \diam(t) \quad\text{for all } t\in\tau_\ell \text{ and }
    T\in\TT_{\lagside}|_{\gamma_\ell} \text{ with } T\subset t.
  \end{align*}
  The constant $C>0$ is independent of $\ell$.
\end{assumption}

We also define the discrete spaces on edges,
\begin{align*}
  Y_{h,\ell} := \{ \psi \in L^2(\gamma) \,:\, \psi|_t \text{ is constant for all } t\in\tau_\ell\},
  \quad \ell=1,\ldots,L,
\end{align*}
and the (global) space for the Lagrangian multiplier
\begin{align*}
  Y_h := \prod_{\ell=1}^L Y_{h,\ell}.
\end{align*}
Based on the previous definitions of decompositions, meshes, and subspaces,
we next introduce the mortar boundary element method.

\subsection{Mortar BEM} \label{sec_mortar}
Let $v = (v_1,\dots,v_N)$ with sufficiently smooth component functions $v_j$ defined on $\Gamma_j$.
We define the piecewise differential operator $\curl_H$ by
\begin{align*}
  \curl_H v:= \sum_{i=1}^N (\curl_{\Gamma_i} v_i)^0 \quad\text{with}\quad \curl_{\Gamma_i} v_i = (\partial_y v_i(x,y),
  -\partial_x v_i(x,y),0)
\end{align*}
and $(\cdot)^0$ being the extension by $0$ onto $\Gamma$.

We need the single layer integral operator
\begin{align*}
  \slp\pphi(x) := \int_\Gamma \frac{\pphi(y)}{|x-y|} \,dy,
\end{align*}
which extends to a continuous operator, mapping $\widetilde\HH_t^{-1/2}(\Gamma)$ to $\HH_t^{1/2}(\Gamma)$.
Here
\begin{align*}
  \HH_t^{1/2}(\Gamma) := \{(v_1,v_2,v_3)\in (H^{1/2}(\Gamma))^3 \,:\, v_3 = 0\}
\end{align*}
and $\widetilde\HH_t^{-1/2}(\Gamma)$ is its dual space.
Furthermore, we define the jumps $[v]$ across interface edges $\gamma_\ell$ by
\begin{align*}
  [v]|_{\gamma_\ell} := \begin{cases}
    v_{\lagside}|_{\gamma_\ell} & \text{if } \gamma_\ell \subset \partial\Gamma, \\
    v_{\lagside}|_{\gamma_\ell} - v_{\morside}|_{\gamma_\ell} & \text{else}.
  \end{cases}
\end{align*}

With the definitions of the bilinear forms
\begin{align*}
  \blf(u,v) &:= \dual{\slp\curl_H u}{\curl_H v}_{\TT} := \sum_{i=1}^N \dual{\slp\curl_H u}{\curl_{\Gamma_i}v}_{\Gamma_i},
  \\
  b(u,\psi) &:= \dual{[v]}\psi_\tau := \sum_{\ell=1}^L \dual{[v]}{\psi}_{\gamma_\ell}
\end{align*}
for all $u,v\in H^{1/2+\eps}(\TT)$ and $\psi\in L^2(\gamma)$ for some $\eps>0$ and the right-hand-side functional
\begin{align*}
  F(v):= \sum_{i=1}^N \dual{f}{v_i}_{\Gamma_i},
\end{align*}
we can state the mortar BEM: Find $(u_h,\phi_h) \in X_h \times Y_h$ such that
\begin{equation}
\begin{split}\label{eq:mortarbem}
\begin{aligned}
  &\blf(u_h,v_h) + b(v_h,\phi_h) \quad&=&\quad F(v_h), \\
  &b(u_h,\psi_h) &=&\quad 0
\end{aligned}
\end{split}
\end{equation}
for all $(v_h,\psi_h)\in X_h \times Y_h$. This formulation admits a unique solution.

\begin{theorem}[{\cite[Theorem~2.1]{mortarBEM}}]\label{thm:mortar}
  Let Assumptions~\ref{ass:interface}--\ref{ass:interfacemesh} hold true.
  There exists a unique solution $(u_h,\phi_h) \in X_h\times Y_h$ of~\eqref{eq:mortarbem}. 
  Assume that the exact solution $u$ of~\eqref{eq:conform:cont} satisfies $u\in \widetilde H^{1/2+r}(\Gamma)$ for some
  $r\in(0,1/2]$. Then, there holds
  \begin{align*}
    \norm{u-u_h}{H^{1/2}(\TT)} \lesssim \lvert\log(\hmin)\rvert^2 \hmax^r \norm{u}{\widetilde H^{1/2+r}(\Gamma)}.
  \end{align*}
\end{theorem}

\begin{remark}
  The work~\cite{mortarBEM} deals with homogeneous boundary conditions on $\partial\Gamma$. However, the analysis can be
  generalized to the present situation, see~\cite{ghh09} for the case of BEM with Lagrangian multipliers.
  In particular,~\cite[Theorem~2.1]{mortarBEM} holds true if we do not impose homogeneous boundary conditions in $X_h$.
\end{remark}

We note that the bilinear form $\blf(\cdot,\cdot)$ is not elliptic due to the fact that piecewise constant functions $c
= (c_1,\dots,c_N)\in X_h$ with $c_i\in\R$ are in the kernel of $\curl_H(\cdot)$.
Therefore, for our analysis we will use a simple stabilization of $\blf(\cdot,\cdot)$ which is similar to the one that
is often used for hypersingular integral equations on closed surfaces.
\begin{lemma}
  Let $0\neq\alpha\in\R$ denote an arbitrary but fixed constant and 
  let $\xi_\ell \in Y_h$ denote the characteristic function on $\gamma_\ell$, i.e., $\xi_\ell|_{\gamma_k} = 
  \delta_{\ell k}$ and define
  \begin{align}\label{eq:def:blfalt}
    \blfalt(u,v) := \blf(u,v)+ \alpha^2 \sum_{\ell=1}^L b(u,\xi_\ell)b(v,\xi_\ell) \quad\text{for all }u,v\in X_h.
  \end{align}
  Then, the variational equation~\eqref{eq:mortarbem} is equivalent to: Find $(u_h,\phi_h)\in X_h\times Y_h$ such that
  \begin{equation}
  \begin{split}\label{eq:mortarbem:alt}
  \begin{aligned}
    &\blfalt(u_h,v_h) + b(v_h,\phi_h) \quad&=&\quad F(v_h) \\
    &b(u_h,\psi_h) &=&\quad 0
  \end{aligned}
  \end{split}
  \end{equation}
  for all $(v_h,\psi_h)\in X_h \times Y_h$.

  Moreover, for $u,v\in X_h$, there holds
  \begin{align}\label{eq:blfalt:ell}
    \norm{v}{H^{1/2}(\TT)}^2 \lesssim \lvert\log(\hmin)\rvert a(v,v) \quad\text{and}\quad a(u,v) \lesssim \lvert\log(\hmin)\rvert^2 \norm{u}{H^{1/2}(\TT)} \norm{v}{H^{1/2}(\TT)}.
  \end{align}
  The involved constants do not depend on $h$.
\end{lemma}

\begin{proof}
  To see the equivalence, we note that $b(u_h,\xi_\ell) =0$ from the second equation of~\eqref{eq:mortarbem}
  resp.~\eqref{eq:mortarbem:alt}, since $\xi_\ell \in Y_h$. Hence, the additional stabilization terms in the definition
  of $\blfalt(\cdot,\cdot)$ always vanish. Note that the solutions $(u_h,\psi_h)$ of~\eqref{eq:mortarbem} 
  and~\eqref{eq:mortarbem:alt} are in fact identical.

  The upper bound in~\eqref{eq:blfalt:ell} follows from the continuity 
  \begin{align*}
    \blf(u,v) \lesssim \lvert\log(\hmin)\rvert^2 \norm{u}{H^{1/2}(\TT)} \norm{v}{H^{1/2}(\TT)}
  \end{align*}
  of $\blf(\cdot,\cdot)$ (see~\cite[Lemma~3.9]{mortarBEM}) and the continuity
  \begin{align*}
    b(u,\psi) \lesssim \lvert\log(\hmin)\rvert^{1/2} \norm{u}{H^{1/2}(\TT)} \norm{\psi}{L^{2}(\gamma)}
  \end{align*}
  of $b(\cdot,\cdot)$ (see~\cite[Lemma~3.14]{mortarBEM}), since
  \begin{align*}
    \sum_{\ell=1}^L b(u,\xi_\ell)b(v,\xi_\ell) \lesssim \lvert\log(\hmin)\rvert \norm{u}{H^{1/2}(\TT)}
    \norm{v}{H^{1/2}(\TT)} \sum_{\ell=1}^L \norm{1}{L^{2}(\gamma_\ell)}^2.
  \end{align*}
  To derive the lower bound in~\eqref{eq:blfalt:ell} we note that the analysis from~\cite{ghh09} and~\cite{mortarBEM} yields
  \begin{align}\label{eq:blfalt:ell:proof1}
    \seminorm{v}{H^{1/2}(\TT)}^2 := \sum_{i=1}^N \seminorm{v_i}{H^{1/2}(\Gamma_i)}^2 \lesssim \blf(v,v) \quad\text{for all } v\in X_h.
  \end{align}
  Moreover, we apply the following result from~\cite[Proposition~3.5]{mortarBEM}, which comes from a discrete
  Poincar\'e-Friedrichs inequality for fractional-order Sobolev spaces proved in~\cite{crBEM}:
  There exists a constant $C>0$ such that for all $\eps\in(0,1/2]$ and any $v\in H^{1/2+\eps}(\TT)$ with
  $v|_{\partial\Gamma}=0$ there holds
  \begin{align}\label{eq:blfalt:ell:proof2}
    C^{-1} \norm{v}{L^2(\Gamma)}^2 \leq \eps^{-1} \seminorm{v}{H^{1/2+\eps}(\TT)}^2 + \sum_{ 
      \substack{\ell\in\{1,\dots,L\} \,: \\
      \gamma_\ell  \text{ is interior edge}} }
    \diam(\gamma_\ell)^{-1-2\eps} \left(\int_{\gamma_\ell} [v] \,ds\right)^2.
  \end{align}
  Let $\widehat\Gamma \supset \Gamma$ denote an extension of $\Gamma$ with $\partial\Gamma \subset \widehat\Gamma$. 
  Moreover, let $\widehat\TT$ denote a subdomain decomposition of $\widehat\Gamma$ with $\TT\subset \widehat\TT$ such
  that the shape regularities of $\TT$ and $\widehat\TT$ are equivalent.
  In particular, $\widehat\TT$ can be chosen such that each boundary edge $\gamma_\ell \subset \partial\Gamma\cap 
  \partial\Gamma_i$ is an interior edge in $\widehat\TT$. Thus,~\eqref{eq:blfalt:ell:proof2} holds true if we replace
  $\Gamma$, resp. $\TT$, with $\widehat\Gamma$, resp. $\widehat\TT$.
  For each $v\in H^{1/2}(\TT)$ we set $\widehat v|_\TT =v$ and $\widehat v|_{\widehat\TT\backslash\TT} :=0$. Then,
  $\widehat v \in H^{1/2}(\widehat\TT)$ and $\norm{\widehat v}{H^{1/2+\eps}(\widehat\TT)} =
  \norm{v}{H^{1/2+\eps}(\TT)}$.
  We infer that
  \begin{align}\label{eq:blfalt:ell:proof3}
  \begin{split}
    \norm{v}{L^2(\Gamma)}^2 &= \norm{\widehat v}{L^2(\widehat\Gamma)}^2 
    \lesssim \eps^{-1} \seminorm{\widehat v}{H^{1/2+\eps}(\widehat\TT)}^2 + \sum_{\ell=1}^L
    \diam(\gamma_\ell)^{-1-2\eps} \left(\int_{\gamma_\ell} [\widehat v] \,ds\right)^2  \\
    &= \eps^{-1} \seminorm{v}{H^{1/2+\eps}(\TT)}^2  + \sum_{\ell=1}^L
    \diam(\gamma_\ell)^{-1-2\eps} b(v,\xi_\ell)^2 
    \\ 
    &\leq \eps^{-1} \seminorm{v}{H^{1/2+\eps}(\TT)}^2  + C' \alpha^2 \sum_{\ell=1}^L b(v,\xi_\ell)^2
  \end{split}
  \end{align}
  with some constant $C'>0$ depending on $\alpha$ and the diameters of $\gamma_\ell$ but not on $\eps$.
  Finally, choosing $\eps = \lvert\log(\hmin)\rvert^{-1}$, the inverse estimate
  \begin{align*}
    \hmin^{\eps} \seminorm{v}{H^{1/2+\eps}(\TT)} \lesssim \seminorm{v}{H^{1/2}(\TT)}
  \end{align*}
  together with~\eqref{eq:blfalt:ell:proof3} and~\eqref{eq:blfalt:ell:proof1} shows ellipticity of
  $\blfalt(\cdot,\cdot)$.
\end{proof}

\subsection{Discretizations}
For real-valued vectors we use bold symbols, e.g., $\xx$. Each vector $\xx\in\R^K$ is uniquely associated to a function
$v\in X_h$ in the following way. Let 
\begin{align*}
  \{\eta_1^{(1)},\dots,\eta_{K_1}^{(1)},\eta_1^{(2)},\dots,\eta_{K_2}^{(2)},\dots,\eta_1^{(N)},\dots,\eta_{K_N}^{(N)} \}
\end{align*}
denote the basis of $X_h$.
For simplicity we use the notation $\eta_1,\dots,\eta_K$ for the basis.
Then, $\xx \in\R^K$ (with $K=\sum_{i=1}^N K_i$) corresponds
to
\begin{align*}
  v = \sum_{i=1}^N \sum_{j=1}^{K_i} \xx_{j + \sum_{k=1}^{i-1}K_k} \eta_j^{(i)} = \sum_{j=1}^K \xx_j \eta_j.
\end{align*}
We define the Galerkin matrix $\Amat \in\R^{K\times K}$ of $\blfalt(\cdot,\cdot)$ as 
\begin{align*}
  \Amat_{jk} := \blfalt(\eta_k,\eta_j) \quad\text{for }j,k=1,\dots,K.
\end{align*}
Let $\{\chi_j^{(\ell)}\}$ denote the basis of $Y_{h,\ell}$ with $\chi_j^{(\ell)}|_{t_k} = \delta_{jk}$ for
$t_k\in\tau_\ell$.
Analogously as before, we write $\chi_1,\dots,\chi_M$ with $M:= \sum_{\ell=1}^L M_\ell := \sum_{\ell=1}^L \# Y_{h,\ell}$
for the corresponding basis of $Y_h$. Then, each $\psi \in Y_h$ can be written as
\begin{align*}
  \psi = \sum_{j=1}^M \yy_j \chi_j \quad\text{for some } \yy\in\R^M.
\end{align*}
We define the matrix $\Bmat\in\R^{M\times K}$ by
\begin{align*}
  \Bmat_{jk} := b(\eta_k,\chi_j) \quad j=1,\dots,M, \, k=1,\dots,K.
\end{align*}
Denoting the right-hand side vector by $\ff\in\R^K$ with $\ff_k := F(\eta_k)$, the formulation~\eqref{eq:mortarbem:alt} is
equivalent to the matrix-vector equation: Find $(\xx,\yy)^T\in\R^{K+M}$ such that
\begin{align*}
  \sysmat
  \begin{pmatrix}
    \xx \\ \yy
  \end{pmatrix}
  :=
  \begin{pmatrix}
    \Amat & \Bmat^T \\
    \Bmat & \bignull
  \end{pmatrix}
  \begin{pmatrix}
    \xx \\ \yy
  \end{pmatrix}
  =
  \begin{pmatrix}
    \ff \\ \bignull
  \end{pmatrix}.
\end{align*}

%%%%%%%%%%%%%%%%%%%%%%%%%%%%%%%%%%%%%%%%%
% Preconditioner
\section{Preconditioning}\label{sec:precond}
In this section we analyze different wirebasket preconditioners for the mortar BEM considered in
Section~\ref{sec:mortar}.
First, we recall results on the MINRES method.

\subsection{Minimal residual method}\label{sec:minres}
Throughout we consider the preconditioned minimal residual method (MINRES) with inner products $\dual\xx\yy_\Pmat :=
\yy^T\Pmat\xx$ 
induced by block-diagonal preconditioners of the form
\begin{align} \label{def:pre}
  \Pmat = \begin{pmatrix}
    \Pmat_\Amat &  \\
    & \Pmat_\Bmat
  \end{pmatrix},
\end{align}
where the blocks $\Pmat_\Amat\in\R^{K\times K}$, $\Pmat_\Bmat\in\R^{M\times M}$ are symmetric and positive definite.
(Here and in the following, empty spaces represent null matrices of appropriate dimensions.)
The preconditioned system then reads
\begin{align*}
  \Pmat^{-1} \sysmat =
  \begin{pmatrix}
    \Pmat_\Amat^{-1} \Amat & \Pmat_\Amat^{-1} \Bmat^T \\
    \Pmat_\Bmat^{-1} \Bmat & 
  \end{pmatrix}
\end{align*}
Furthermore, define the matrix
\begin{align*}
  \widetilde\sysmat := 
  \begin{pmatrix}
    \widetilde\Amat & \widetilde\Bmat^T \\
    \widetilde\Bmat & 
  \end{pmatrix} 
  :=
  \begin{pmatrix}
    \Pmat_\Amat^{-1/2} \Amat \Pmat_\Amat^{-1/2} & \Pmat_\Amat^{-1/2} \Bmat^T \Pmat_\Bmat^{-1/2} \\
    \Pmat_\Bmat^{-1/2} \Bmat \Pmat_\Amat^{-1/2} &
  \end{pmatrix}
  =  \Pmat^{-1/2} \sysmat \Pmat^{-1/2}.
\end{align*}
We note that there holds $\spec(\widetilde\sysmat) = \spec(\Pmat^{-1}\sysmat)$ for the respective spectra.
Let 
\begin{align}\label{eq:def:ev_sv}
  \Evmin \leq \Evmax
  \quad\text{and}\quad
  \Sigma_1 \leq  \dots \leq \Sigma_m
\end{align}
denote, respectively, the extremal eigenvalues of $\widetilde\Amat$ and
the nonzero singular values of $\widetilde\Bmat$. Of course, they are all positive.
We also define the condition number
\[
   \kappa(\widetilde\sysmat)
   :=
   \max\{|\lambda|;\; \lambda\in\spec(\widetilde\sysmat)\}/\min\{|\lambda|;\; \lambda\in\spec(\widetilde\sysmat)\}.
\]
The following is a well-established result, see, e.g.,~\cite{wfs95}.

\begin{proposition}\label{prop:minres}
  Denote by $\rr^{(k)} := \Pmat^{-1}(\ff-\sysmat \xx^{(k)})$ the residual of the
  $k$-th preconditioned MINRES iteration $\xx^{(k)}$ with inner product $\dual\cdot\cdot_\Pmat$.
  Then there holds
  \begin{align*}
    \frac{\norm{\rr^{(k)}}\Pmat^2}{\norm{\rr^{(0)}}\Pmat^2} \leq 2
    \left(\frac{\kappa(\widetilde\sysmat)-1}{\kappa(\widetilde\sysmat)+1} \right)^k.
  \end{align*}
  so that the number of preconditioned MINRES iterations, required to reduce the initial residual
  to a certain percentage, is bounded by $O(\kappa(\widetilde\sysmat))$.
\end{proposition}

Bounds for the spectrum of $\widetilde\sysmat$ can be specified in terms of the eigenvalues of $\widetilde\Amat$ and
singular values of $\widetilde\Bmat$.

\begin{proposition}[{\cite[Lemma~2.1]{rusWin92}}]\label{prop:evsysmat}
  There holds
  \begin{align*}
  \begin{split}
    \spec(\Pmat^{-1}\sysmat) \subseteq 
    &[\tfrac12(\Evmin-\sqrt{\Evmin^2 + 4\Sigma_m^2}, \tfrac12(\Evmax - \sqrt{\Evmax^2+4\Sigma_1^2})] \\
    &\qquad \cup [\Evmin,\tfrac12(\Evmax+\sqrt{\Evmax^2 + 4\Sigma_m^2})]
  \end{split}
  \end{align*}
  with $\Evmin$, $\Evmax$, $\Sigma_1$, $\Sigma_m$ being the numbers from~\eqref{eq:def:ev_sv}.
\end{proposition}

\subsection{Preconditioner and main results}
Our preconditioning technique is based on an initial decomposition of $X_h$ into wirebasket components
related with the coarse mesh $\TT$ and the remainder. Then, individual preconditioners are applied to
the three spaces of wirebasket and interior components and the Lagrangian multiplier.

\subsubsection{Wirebasket splitting.}
For the initial decomposition of $X_h$, we define for each $v_i \in X_{h,i}$ the unique representation
\begin{align} \label{splitting}
  v_i = v_{i,1} + v_{i,0} \quad\text{with}\quad v_{i,0}\in X_{i,0} := \{w\in X_{h,i} \,:\, w|_{\partial\Gamma_i}=0 \}
  \text{ and } v_{i,1}:=v_i-v_{i,0},
\end{align}
and the preconditioning forms $d_j: X_h\times X_h \to \R$ ($j=1,2,3$) defined by
\begin{subequations}\label{eq:defprecform}
\begin{align}
  d_1(u,v) &:=\sum_{i=1}^N \dual{u_{i,1}|_{\partial\Gamma_i}}{v_{i,1}|_{\partial\Gamma_i}}_{L^2(\partial\Gamma_i)} 
  + \sum_{i=1}^N \dual{\hyp_i u_{i,0}}{v_{i,0}}_{\Gamma_i}, \\
  d_2(u,v) &:=\sum_{i=1}^N \dual{u_{i,1}|_{\partial\Gamma_i}}{v_{i,1}|_{\partial\Gamma_i}}_{L^2(\partial\Gamma_i)} 
  + \frac{1}{\lvert\log(\hmin)\rvert^2} \sum_{i=1}^N \dual{\hyp_i u_{i,0}}{v_{i,0}}_{\Gamma_i}, \\
  d_3(u,v) &:=\sum_{i=1}^N \lvert\log(\hmin)\rvert
  \dual{u_{i,1}|_{\partial\Gamma_i}}{v_{i,1}|_{\partial\Gamma_i}}_{L^2(\partial\Gamma_i)} 
  + \sum_{i=1}^N \dual{\hyp_i u_{i,0}}{v_{i,0}}_{\Gamma_i}.
\end{align}
\end{subequations}
Here, $\hyp_i$ is the hypersingular integral operator associated with the subdomain $\Gamma_i$.
Note that, with $\slp_i$ being the simple-layer integral operator associated with $\Gamma_i$, we have
\begin{align*}
  \dual{\slp_i \curl_{\Gamma_i}v_{i,0}}{\curl_{\Gamma_i} v_{i,0}}_{\Gamma_i} 
= \dual{\hyp_i v_{i,0}}{v_{i,0}}_{\Gamma_i} \simeq \norm{v_{i,0}}{\widetilde H^{1/2}(\Gamma_i)}^2.
\end{align*}
Definition~\eqref{eq:defprecform} provides, up to logarithmic terms, stable splittings.
\begin{lemma}\label{lem:wbsplitting}
  For all $v \in X_h$ there holds
%  \begin{subequations}
  \begin{align*}
    \lvert\log(\hmin)\rvert^{-3} d_1(v,v) &\lesssim \blfalt(v,v) \lesssim \lvert\log(\hmin)\rvert^2 d_1(v,v), \\
    \lvert\log(\hmin)\rvert^{-2} d_2(v,v) &\lesssim \blfalt(v,v) \lesssim \lvert\log(\hmin)\rvert^2 d_2(v,v), \\
    \lvert\log(\hmin)\rvert^{-3} d_3(v,v) &\lesssim \blfalt(v,v) \lesssim \lvert\log(\hmin)\rvert d_3(v,v).
  \end{align*}
%  \end{subequations}
\end{lemma}
\noindent
A proof of Lemma~\ref{lem:wbsplitting} is given in Section~\ref{sec:main:proof}.

\subsubsection{Preconditioner for $\Amat$.}
We now consider a preconditioner for the matrix $\Amat$ that corresponds to the bilinear
form $\blfalt(\cdot,\cdot)$ on $X_h\times X_h$.
Having performed the initial decomposition of $X_h$ into wirebasket and interior components,
Lemma~\ref{lem:wbsplitting} and the structure of the bilinear forms $d_j$ defined
by \eqref{eq:defprecform} show that it suffices to provide preconditioners for the
$L^2(\partial\Gamma_i)$ terms and the terms involving the hypersingular integral operator.
We use, respectively, a simple diagonal preconditioner and an arbitrary preconditioner
for the hypersingular integral operator in the conforming case, see,
e.g.,~\cite{amcl03,amt99,swOpPrec,transtep96}.

In the following, let $\Pmat_{\hyp_i}$ denote such a preconditioner for the hypersingular integral operator $\hyp_i$ with
constants $\evmin^{(i)}, \evmax^{(i)}$ such that 
\begin{align}\label{eq:defPrecHyp}
  \evmin^{(i)} \xx^T \Pmat_{\hyp_i} \xx \leq \dual{\hyp_i v_{i,0}}{v_{i,0}}_{\Gamma_i} \leq
  \evmax^{(i)} \xx^T \Pmat_{\hyp_i} \xx
\end{align}
for all $v_{i,0} \in X_{i,0}$ with $v_{i,0} = \sum_{z_j\in\KK_i^0} \xx_j \eta_j^{(i)}$.
Furthermore, let $\Pmat_{\partial\Gamma_i}$ denote a preconditioner with
\begin{align}\label{eq:defPrecL2}
  \mu_\mathrm{min}^{(i)}\yy^T \Pmat_{\partial\Gamma_i} \yy \leq
  \norm{v_{i,1}|_{\partial\Gamma_i}}{L^2(\partial\Gamma_i)}^2 \leq 
  \mu_\mathrm{max}^{(i)}\yy^T \Pmat_{\partial\Gamma_i} \yy
\end{align}
for all $v_{i,1}\in X_{i,1}$ with $v_{i,1} = \sum_{z_j \in \KK_i^1} \yy_j \eta_j^{(i)}$. 
Define the preconditioner $\Pmat^{(i)}\in\R^{K_i\times K_i}$ for the $i$-th subdomain by 
\begin{align*}
  \Pmat^{(i)} := 
  \begin{cases}
    \begin{pmatrix}
      \Pmat_{\partial\Gamma_i} & \\
      & \Pmat_{\hyp_i}
    \end{pmatrix}
    & \text{if } d_1(\cdot,\cdot)\text{ is used}, \\
    \begin{pmatrix}
      \Pmat_{\partial\Gamma_i} & \\
      & \lvert\log(\hmin)\rvert^{-2}\Pmat_{\hyp_i}
    \end{pmatrix}
    & \text{if } d_2(\cdot,\cdot)\text{ is used}, \\
    \begin{pmatrix}
      \lvert\log(\hmin)\rvert
      \Pmat_{\partial\Gamma_i} & \\
      & \Pmat_{\hyp_i}
    \end{pmatrix}
    & \text{if } d_3(\cdot,\cdot)\text{ is used},
  \end{cases}
\end{align*}
and the overall preconditioner $\Pmat_\Amat$ for the matrix $\Amat$, corresponding to the bilinear
form $\blfalt(\cdot,\cdot)$ on $X_h$, by
\begin{align*}
  \Pmat_\Amat :=
  \begin{pmatrix}
    \Pmat^{(1)} & & \\
    & \ddots & \\
    & & \Pmat^{(N)}
  \end{pmatrix}.
\end{align*}
For the last two definitions we have assumed an appropriate order of the degrees of freedom in $X_{h,i}$.
The logarithmic terms in the definition of $\Pmat^{(i)}$ stem from the logarithmic perturbations in the
definition~\eqref{eq:defprecform} of $d_j(\cdot,\cdot)$.
In the remainder of this work we will refer to\\
\centerline{``\emph{Case~$j$}''\quad if $d_j(\cdot,\cdot)$ is used in the definition of $\Pmat_\Amat$ ($j=1,2,3$).}

Our main result concerning the preconditioning of $\Amat$ is as follows.

\begin{theorem}\label{thm:main}
  Set 
  \begin{align*}
    \evmin := \min\{\evmin^{(1)},\mu_\mathrm{min}^{(1)},\dots,\evmin^{(N)},\mu_\mathrm{min}^{(N)}\} \text{ and } 
    \evmax :=  \max\{\evmax^{(1)},\mu_\mathrm{max}^{(1)},\dots,\evmax^{(N)},\mu_\mathrm{max}^{(N)}\}.
  \end{align*}
  Then, there holds for all $\xx\in\R^{K}$
  \begin{align*}
    \begin{cases}
      \lvert\log(\hmin)\rvert^{-3} \evmin \xx^T \Pmat_\Amat \xx \lesssim \xx^T \Amat \xx 
      \lesssim \lvert\log(\hmin)\rvert^2\evmax \xx^T \Pmat_\Amat \xx  & \text{for \emph{Case~1}}, \\
      \lvert\log(\hmin)\rvert^{-2} \evmin \xx^T \Pmat_\Amat \xx \lesssim \xx^T \Amat \xx 
      \lesssim \lvert\log(\hmin)\rvert^2\evmax \xx^T \Pmat_\Amat \xx & \text{for \emph{Case~2}}, \\
     \lvert\log(\hmin)\rvert^{-3} \evmin \xx^T \Pmat_\Amat \xx \lesssim \xx^T \Amat \xx 
     \lesssim \lvert\log(\hmin)\rvert\evmax \xx^T \Pmat_\Amat \xx & \text{for \emph{Case~3}}.
  \end{cases}
  \end{align*}
  Therefore, the condition number of $\widetilde\Amat$ is bounded by
  \begin{align*}
    \kappa(\widetilde\Amat) \lesssim \lvert\log(\hmin)\rvert^{\beta} \frac{\evmax}{\evmin}
  \end{align*}
  with $\beta = 5$ in \emph{Case~1} and $\beta=4$ in \emph{Cases~2,3}.
\end{theorem}
\begin{proof}
  The proof follows directly from Lemma~\ref{lem:wbsplitting} and Assumptions~\eqref{eq:defPrecHyp},
  ~\eqref{eq:defPrecL2} on the preconditioners.
\end{proof}

\subsubsection{Final preconditioner for the full matrix $\sysmat$.}
In order to define the preconditioner $\Pmat$ \eqref{def:pre} for the full matrix $\sysmat$ we
assume that we have a matrix $\Pmat_\Bmat\in\R^{M\times M}$ such that there exist numbers
$\sigma_\mathrm{min},\sigma_\mathrm{max} >0$ with
\begin{align}\label{eq:specteq:Lagrange}
  \sigma_\mathrm{min} \yy^T\Pmat_\Bmat\yy \leq \norm{\psi}{L^2(\gamma)}^2 \leq
  \sigma_\mathrm{max} \yy^T\Pmat_\Bmat\yy
\end{align}
for all $\psi = \sum_{m=1}^M \yy_m \chi_m \in Y_h$.
Below, we will select $\Pmat_\Bmat$ to be diagonal with or without logarithmic scaling.

To provide bounds for the spectrum of $\widetilde\sysmat$ by means of Proposition~\ref{prop:evsysmat}
it remains to bound the singular values $\Sigma_1,\dots,\Sigma_m$ of the matrix $\widetilde\Bmat$.

\begin{lemma}\label{lem:svbounds}
  Let $0<\Sigma_1 \leq \dots \leq\Sigma_m$ denote the nonzero singular values of the matrix $\widetilde\Bmat$
  and let $\evmin,\evmax$ be defined as in Theorem~\ref{thm:main}.
  Then,
  \begin{align*}
    \begin{cases}
      \evmin \sigma_\mathrm{min} \lvert\log(\hmin)\rvert^{-2} \lesssim \Sigma_1^2 \leq \Sigma_m^2 \lesssim
      \evmax  \sigma_\mathrm{max}  &\text{for \emph{Case~1}}, \\
      \evmin \sigma_\mathrm{min} \lvert\log(\hmin)\rvert^{-1} \lesssim \Sigma_1^2 \leq \Sigma_m^2 \lesssim
      \evmax  \sigma_\mathrm{max} &\text{for \emph{Case~2}}, \\
      \evmin \sigma_\mathrm{min} \lvert\log(\hmin)\rvert^{-2} \lesssim \Sigma_1^2 \leq \Sigma_m^2 \lesssim
      \evmax  \sigma_\mathrm{max} \lvert\log(\hmin)\rvert^{-1} &\text{for \emph{Case~3}}.
    \end{cases}
  \end{align*}
\end{lemma}

A proof of Lemma~\ref{lem:svbounds} will be given in Section~\ref{sec:main:proof}.

Now, let $\Mmat\in\R^{M\times M}$ denote the $L^2(\gamma)$ mass matrix, i.e.,
\begin{align*}
  \Mmat_{jk} := \dual{\chi_j}{\chi_k}_\gamma \quad\text{for } j,k=1,\dots,M.
\end{align*}
Obviously, $\Mmat$ is diagonal and
\begin{align*}
  \norm\psi{L^2(\gamma)}^2 = \yy^T\Mmat\yy \quad\text{for all } \psi = \sum_{m=1}^M \yy_m \chi_m \in Y_h.
\end{align*}
The main result of our paper is the next theorem. Its proof is immediate by combining
the previously established estimates, namely Theorem~\ref{thm:main} and Lemma~\ref{lem:svbounds},
together with the general results provided by Propositions~\ref{prop:minres} and~\ref{prop:evsysmat}.

\begin{theorem}\label{thm:cond}
  Let $\evmin,\evmax$ be defined as in Theorem~\ref{thm:main} and let $\sigma_\mathrm{min}, \sigma_\mathrm{max}>0$
  be the numbers from \eqref{eq:specteq:Lagrange}.
  Then the spectrum of the preconditioned matrix has a superset like
  $\spec(\widetilde\sysmat) \subseteq [-a,-b] \cup [c,d]$ with numbers $a,b,c,d>0$ that satisfy the following
  estimates.
  \begin{itemize}
    \item \emph{Case~1}: If $\Pmat_\Bmat = \lvert\log(\hmin)\rvert^{-1} \Mmat$,
      then $\sigma_\mathrm{min} = \sigma_\mathrm{max} = \lvert\log(\hmin)\rvert$ and
      \begin{align*}
        \evmin/\max\{\evmax,\evmax^{1/2}\} \lvert\log(\hmin)\rvert^{-3} \lesssim b &\leq a \lesssim
        \evmax^{1/2} \lvert\log(\hmin)\rvert^{1/2}, \\
        \evmin \lvert\log(\hmin)\rvert^{-3} \lesssim c &\leq d  \lesssim \max\{\evmax,\evmax^{1/2}\}\lvert\log(\hmin)\rvert^{2}.
      \end{align*}
    \item \emph{Case~2}: If $\Pmat_\Bmat = \lvert\log(\hmin)\rvert^{-1} \Mmat$,
      then $\sigma_\mathrm{min} = \sigma_\mathrm{max} = \lvert\log(\hmin)\rvert$ and
      \begin{align*}
        \evmin/\max\{\evmax,\evmax^{1/2}\} \lvert\log(\hmin)\rvert^{-2} \lesssim b &\leq a \lesssim
        \evmax^{1/2} \lvert\log(\hmin)\rvert^{1/2}, \\
        \evmin \lvert\log(\hmin)\rvert^{-2} \lesssim c &\leq d  \lesssim \max\{\evmax,\evmax^{1/2}\}\lvert\log(\hmin)\rvert^{2}.
      \end{align*}
    \item \emph{Case~3}: If $\Pmat_\Bmat = \Mmat$, then $\sigma_\mathrm{min} = \sigma_\mathrm{max} = 1$ and
      \begin{align*}
        \evmin/\max\{\evmax,\evmax^{1/2}\} \lvert\log(\hmin)\rvert^{-3}\lesssim b &\leq a 
        \lesssim (\evmax)^{1/2} \lvert\log(\hmin)\rvert^{-1/2}, \\
        \evmin \lvert\log(\hmin)\rvert^{-3} \lesssim c &\leq d  \lesssim \max\{\evmax,\evmax^{1/2}\}\lvert\log(\hmin)\rvert.
      \end{align*}
  \end{itemize}
  Therefore, the condition number of $\widetilde\sysmat$ is bounded by
  \begin{align*}
    \kappa(\widetilde\sysmat) \lesssim
    \lvert\log(\hmin)\rvert^{\beta} \max\{\evmax^{1/2}, \evmax^2\} / \evmin
  \end{align*}
  with $\beta = 5$ in \emph{Case~1} and $\beta=4$ in \emph{Cases~2,3}.
  Furthermore, the number of preconditioned MINRES iterations, required to reduce the relative residual to
  a certain threshold, is bounded like the condition number in the respective case.
\end{theorem}

% -----------------------------------------------------------------------------------------
% PROOFS
% -----------------------------------------------------------------------------------------
\subsection{Proofs and technical details}\label{sec:main:proof}
For the proof of Lemma~\ref{lem:wbsplitting} we need a trace inequality and an inverse estimate, which are given in the
following two lemmas.
\begin{lemma}[{\cite[Lemma~4.3]{ghh09}}]\label{lem:traceineq}
  Let $R\subset \R^2$ be a bounded Lipschitz domain. There exists a constant $C>0$ such that for all $\eps \in (0,1/2)$
  holds
  \begin{align*}
    \norm{v}{L^2(\partial R)} \leq C \eps^{-1/2} \norm{v}{H^{1/2+\eps}(R)} \quad\text{for all } v\in H^{1/2+\eps}(R).
  \end{align*}
\end{lemma}

\begin{lemma}[{\cite[Lemma~4]{HS2001}}]\label{lem:traceineq:multilevel}
  For a function $v_i\in X_{i,h}$ with splitting \eqref{splitting}, $v_i = v_{i,0} + v_{i,1}$, there holds
  \begin{align*}
    \norm{v_{i,0}}{\widetilde H^{1/2}(\Gamma_i)} \lesssim \lvert\log(\hmin_i)\rvert \norm{v_i}{H^{1/2}(\Gamma_i)},
    \quad i=1,\ldots,N.
  \end{align*}
\end{lemma}

The proof of~\cite[Lemma~4]{HS2001} uses~\cite[Lemma~4.5]{dsw94}.
An alternative proof of Lemma~\ref{lem:traceineq:multilevel} which utilizes multilevel norms is given
in~\cite[Theorem~3.6]{HJM12}.

\begin{proof}[Proof of Lemma~\ref{lem:wbsplitting}]
  We start with a proof of the upper bound. 
  Let $v = v^{(0)} + v^{(1)} \in X_h$ with $v_i^{(0)} := v_{i,0}$ and $v_i^{(1)} := v_{i,1}$.
  Application of the triangle inequality, boundedness~\eqref{eq:blfalt:ell} of the bilinear form $\blfalt(\cdot,\cdot)$, and equivalence
  $\dual{\slp\cdot}{\cdot}_\Gamma \simeq \norm{\cdot}{\widetilde \HH_t^{-1/2}(\Gamma)}^2$ together with the estimate 
  $\norm{\cdot}{\widetilde \HH_t^{-1/2}(\Gamma)} \lesssim \norm{\cdot}{\widetilde \HH_t^{-1/2}(\TT)}$ for fractional-order
  Sobolev spaces, leads to
  \begin{align*}
    \blfalt(v,v) &\lesssim \blfalt(v^{(0)},v^{(0)}) + \blfalt(v^{(1)},v^{(1)}) 
    \lesssim  \blfalt(v^{(0)},v^{(0)}) + \lvert\log(\hmin)\rvert^2 \norm{v^{(1)}}{H^{1/2}(\TT)}^2 \\
    &\lesssim \sum_{i=1}^N \norm{\curl_{\Gamma_i} v_{i,0}}{\widetilde\HH_t^{-1/2}(\Gamma_i)}^2 + \lvert\log(\hmin)\rvert^2
    \sum_{i=1}^N \norm{v_{i,1}}{H^{1/2}(\Gamma_i)}^2
  \end{align*}
  Then, $\dual{\slp_i\cdot}{\cdot}_{\Gamma_i} \simeq \norm{\cdot}{\widetilde \HH_t^{-1/2}(\Gamma_i)}^2$ and 
  $\dual{\slp_i \curl_{\Gamma_i} u_{i,0}}{\curl_{\Gamma_i}v_{i,0}}_{\Gamma_i} = \dual{\hyp_i
    u_{i,0}}{v_{i,0}}_{\Gamma_i}$ for all $u,v\in X_h$
  show
  \begin{align*}
    \blfalt(v,v) &\lesssim \sum_{i=1}^N \dual{\hyp_i v_{i,0}}{v_{i,0}}_{\Gamma_i} + \lvert\log(\hmin)\rvert^2
    \sum_{i=1}^N \norm{v_{i,1}}{H^{1/2}(\Gamma_i)}^2.
  \end{align*}
  Let $\widetilde\Gamma_i$ denote a closed extension of the subdomain
  $\Gamma_i$ and let $\widetilde\TT_i$ denote an extension of the mesh $\TT_i$ such that the shape-regularities of the
  meshes $\widetilde\TT_i$ and $\TT_i$ are equivalent.
  For $z_j \in \KK_i^0$ we set $\widetilde\eta_j^{(i)} := \eta_j^{(i)}$ and for
  $z_j\in\KK_i^1$ we define $\widetilde\eta_j^{(i)}$ as the (bi-)linear function with
  $\widetilde\eta_j^{(i)}(z_k) = \delta_{jk}$ for all nodes $z_k$ of the mesh $\widetilde\TT_i$.
  Hence, $\widetilde\eta_j^{(i)}|_{\Gamma_i} = \eta_j^{(i)}$.
  For an arbitrary function $v_i = \sum_{z_j\in\KK_i} \xx_j \eta_j^{(i)} \in X_{h,i}$ we define its extension $\widetilde
  v_i$ as 
  \begin{align*}
    \widetilde v_i := \sum_{z_j\in\KK_i} \xx_j \widetilde\eta_j^{(i)} \in \widetilde X_{h,i}.
  \end{align*}
  By the properties of the $H^{1/2}$- and $\widetilde H^{1/2}$-norms, we have
  \begin{align*}
    \norm{v_{i,1}}{H^{1/2}(\Gamma_i)}^2 \lesssim \norm{\widetilde v_{i,1}}{\widetilde H^{1/2}(\widetilde\Gamma_i)}^2.
    %\simeq \norm{\widetilde v_{i,1}}{\widetilde H^{1/2}(\widetilde\Gamma_i)}^2.
  \end{align*}
  Set $\omega_k := \supp(\widetilde\eta_k^{(i)})$. We note that there exists a constant $C_\mathrm{col}>0$ that
  depends only on the shape-regularity of the mesh $\TT_i$ such that
  \begin{align*}
    \norm{\widetilde v_{i,1}}{\widetilde H^{1/2}(\widetilde\Gamma_i)}^2 \leq 
    C_\mathrm{col} \sum_{z_k\in\KK_i^1} \norm{\xx_k \widetilde\eta_k^{(i)}}{\widetilde
      H^{1/2}(\omega_k)}^2.
  \end{align*}
  With 
  \begin{align*}
    \norm{\widetilde \eta_k^{(i)}}{\widetilde H^{1/2}(\omega_k)}^2 \simeq \diam(\omega_k) \simeq
    \norm{\widetilde\eta_k^{(i)}|_{\partial\Gamma_i}}{L^2(\omega_k\cap\partial\Gamma_i)}^2
    = \norm{\eta_k^{(i)}|_{\partial\Gamma_i}}{L^2(\omega_k\cap\partial\Gamma_i)}^2
  \end{align*}
  and the locality of the $L^2$-norms we further deduce 
  \begin{align*}
    \norm{\widetilde v_{i,1}}{\widetilde H^{1/2}(\widetilde\Gamma_i)}^2 &\lesssim
    \sum_{z_k\in\KK_i^1} \norm{\xx_k \widetilde\eta_k^{(i)}}{\widetilde H^{1/2}(\omega_k)}^2
    \lesssim \sum_{z_k\in\KK_i^1}
    \norm{\xx_k \eta_k^{(i)}|_{\partial\Gamma_i}}{L^2(\omega_k\cap\partial\Gamma_i)}^2
    \simeq \norm{v_{i,1}|_{\partial\Gamma_i}}{L^2(\partial\Gamma_i)}^2.
  \end{align*}
  Thus, altogether we have
  \begin{align*}
    \blfalt(v,v) &\lesssim \sum_{i=1}^N \dual{\hyp_i v_{i,0}}{v_{i,0}}_{\Gamma_i} + \lvert\log(\hmin)\rvert^2
    \sum_{i=1}^N \norm{v_{i,1}|_{\partial\Gamma_i}}{L^2(\partial\Gamma_i)}^2,
  \end{align*}
  which proves the upper bounds.

  For the lower bounds, we use Lemma~\ref{lem:traceineq} with $R=\Gamma_i$ and
  $\eps= \lvert\log(\hmin)\rvert^{-1}$ (for $\hmin$ small enough).
  Together with an inverse inequality this gives
  \begin{align*}
    \norm{v_{i,1}|_{\partial\Gamma_i}}{L^2(\partial\Gamma_i)}^2 \lesssim \lvert\log(\hmin)\rvert \norm{v_i}{H^{1/2}(\Gamma_i)}^2.
  \end{align*}
  Using the norm equivalence $\norm{\cdot}{\widetilde H^{1/2}(\Gamma_i)}^2\simeq \dual{\hyp_i \cdot}{\cdot}_{\Gamma_i}$
  and Lemma~\ref{lem:traceineq:multilevel} shows that
  \begin{align*}
    \dual{\hyp_i v_{i,0}}{v_{i,0}}_{\Gamma_i} \simeq  \norm{v_{i,0}}{\widetilde H^{1/2}(\Gamma_i)}^2
    \lesssim \lvert\log(\hmin)\rvert^2 \norm{v_i}{H^{1/2}(\Gamma_i)}^2.
  \end{align*}
  Combining the previous relations and summing over $i=1,\dots,N$ proves 
  \begin{align*}
    d_1(v,v) &= \sum_{i=1}^N \dual{\hyp_i v_{i,0}}{v_{i,0}}_{\Gamma_i} + 
    \sum_{i=1}^N \norm{v_{i,1}|_{\partial\Gamma_i}}{L^2(\partial\Gamma_i)}^2 \leq \lvert\log(\hmin)\rvert^2
    \norm{v}{H^{1/2}(\TT)}^2,\\
    d_2(v,v) &= \sum_{i=1}^N \lvert\log(\hmin)\rvert^{-2} \dual{\hyp_i v_{i,0}}{v_{i,0}}_{\Gamma_i} + 
    \sum_{i=1}^N \norm{v_{i,1}|_{\partial\Gamma_i}}{L^2(\partial\Gamma_i)}^2 \leq \lvert\log(\hmin)\rvert
    \norm{v}{H^{1/2}(\TT)}^2, \\
    d_3(v,v) &= \sum_{i=1}^N \dual{\hyp_i v_{i,0}}{v_{i,0}}_{\Gamma_i} + \lvert\log(\hmin)\rvert
    \sum_{i=1}^N \norm{v_{i,1}|_{\partial\Gamma_i}}{L^2(\partial\Gamma_i)}^2 \leq \lvert\log(\hmin)\rvert^2
    \norm{v}{H^{1/2}(\TT)}^2.
  \end{align*}
  Hence, by applying the ellipticity of $\blfalt(\cdot,\cdot)$ from Theorem~\ref{thm:mortar},
  this shows the lower bounds.
\end{proof}

\begin{proof}[{Proof of Lemma~\ref{lem:svbounds}}]
  Note that the nonzero singular values of $\widetilde\Bmat$ are given by the square roots of the eigenvalues of the
  matrix $\widetilde\Bmat\widetilde\Bmat^T$, since $\widetilde\Bmat$ has full (row) rank.
  Furthermore, we note that the smallest and largest singular values are given, respectively,
  by the minimum and maximum of the term
  \begin{align*}
    \max_{\yy\in\R^K} \frac{b(v,\psi)}{\norm{\yy}{\Pmat_\Amat}\norm{\xx}{\Pmat_\Bmat}} \quad\text{with } v =
    \sum_{k=1}^K \yy_k \eta_k \text{ and } \psi = \sum_{m=1}^M \xx_m \chi_m.
  \end{align*}
  We start with the upper bound.
  By the Cauchy-Schwarz and triangle inequalities we have
  \begin{align*}
    b(v,\psi)^2 \leq 2 \norm{\psi}{L^2(\gamma)}^2 \sum_{i=1}^N \norm{v_i}{L^2(\partial\Gamma_i)}^2 
    \leq \norm{\psi}{L^2(\gamma)}^2 \lvert\log(\hmin)\rvert^{M_j} d_j(v,v)
  \end{align*}
  with $M_1=M_2=0$ and $M_3=-1$.
  This together with $d_j(v,v)\lesssim \evmax \yy^T \Pmat_\Amat \yy$ and
  $\norm{\psi}{L^2(\gamma)}^2 \leq \sigma_\mathrm{max} \xx^T\Pmat_\Bmat\xx$ from~\eqref{eq:specteq:Lagrange} 
  proves the upper bound.

  For the lower bound, we use the proof of Lemma~\ref{lem:wbsplitting} to see that
  \begin{align*}
    \evmin \yy^T \Pmat_\Amat \yy \lesssim d_j(v,v) \lesssim \lvert\log(\hmin)\rvert^{m_j} \norm{v}{H^{1/2}(\TT)}^2
  \end{align*}
  with $m_1 = m_3 = -2$ and $m_2 = -1$.
  This leads to the estimate
  \begin{align*}
    \max_{\yy\in\R^K \backslash \{0\}} \frac{b(v,\psi)}{\norm{\yy}{\Pmat_\Amat}\norm{\xx}{\Pmat_\Bmat}} \gtrsim
    (\evmin)^{1/2} \lvert\log(\hmin)\rvert^{-m_j/2} \max_{v\in X_h \backslash \{0\}} \frac{b(v,\psi)}{\norm{v}{H^{1/2}(\TT)}\norm{\xx}{\Pmat_\Bmat}}.
  \end{align*}
  By using $\norm{\psi}{L^2(\gamma)}^2 \geq \sigma_\mathrm{min} \xx^T\Pmat_\Bmat\xx$
  from~\eqref{eq:specteq:Lagrange} and the discrete \emph{inf-sup}
  condition 
  \begin{align*}
    \sup_{0\neq v\in X_h} \frac{b(v,\psi)}{\norm{v}{H^{1/2}(\TT)}} \geq \beta \norm{\psi}{L^2(\gamma)} \quad\text{for
    all } \psi \in Y_h
  \end{align*}
  from~\cite[Lemma~3.12]{mortarBEM} (with constant $\beta>0$ independent of $h$) we conclude the lower bound.
\end{proof}

%%%%%%%%%%%%%%%%%%%%%%%%%%%%%%%%%%%%%%%%%
% Examples
%%%%%%%%%%%%%%%%%%%%%%%%%%%%%%%%%%%%%%%%%
% Examples
%%%%%%%%%%%%%%%%%%%%%%%%%%%%%%%%%%%%%%%%%
\section{Numerical Examples}\label{sec:examples}
In this section we present numerical examples in which we compare the different behaviors of the preconditioned systems
induced by the preconditioning forms $d_1(\cdot,\cdot),d_2(\cdot,\cdot),d_3(\cdot,\cdot)$. 
Note that block $\Pmat_\Amat$ of the preconditioner $\Pmat$ is determined by the preconditioning forms $d_j(\cdot,\cdot)$. 
For the second block $\Pmat_\Bmat$ we choose, up to a possible logarithmic term, the (diagonal) mass matrix $\Mmat$ for the
Lagrangian multiplier space.
We distinguish the following cases (with corresponding numbers $\sigma_\mathrm{min},\sigma_\mathrm{max}$
according to Theorem~\ref{thm:cond}, cf.~\eqref{eq:specteq:Lagrange}).

\[
  \begin{array}{lll}
   \text{\emph{Case 1a)}} & \Pmat_\Bmat = \lvert\log(\hmin)\rvert^{-1} \Mmat,
                   & \sigma_\mathrm{min} = \sigma_\mathrm{max} = \lvert\log(\hmin)\rvert,\\
   \text{\emph{Case 1b)}} & \Pmat_\Bmat = \Mmat,
                   & \sigma_\mathrm{min} = \sigma_\mathrm{max} = 1,\\
   \text{\emph{Case 2)}}  & \Pmat_\Bmat = \lvert\log(\hmin)\rvert^{-1} \Mmat,
                   & \sigma_\mathrm{min} = \sigma_\mathrm{max} = \lvert\log(\hmin)\rvert,\\
   \text{\emph{Case 3)}}  & \Pmat_\Bmat = \Mmat,
                   & \sigma_\mathrm{min} = \sigma_\mathrm{max} = 1.
  \end{array}
\]
Note that \emph{Case 1a}, \emph{Case 2}, \emph{Case 3} correspond to the bounds obtained in Theorem~\ref{thm:cond}, whereas, at least
theoretically, we would expect worse bounds for \emph{Case 1b}.
Moreover, we compare the results to a simple diagonal preconditioner with
\begin{align*}
  (\Pmat_\Amat)_{jk} = \Amat_{jj}\delta_{jk} \text{ and } (\Pmat_\Bmat)_{jk} = \Bmat_{jj}\delta_{jk},
\end{align*}
where $\delta_{jk}$ denotes the Kronecker delta symbol.
In the figures and tables below, we refer to this preconditioner as \emph{diag}.

Throughout, we use the MINRES algorithm, see Section~\ref{sec:minres}, to solve the discrete system. We stop the
algorithm if the relative residual in the $k$-th step satisfies
\begin{align*}
  \frac{\norm{\rr^{(k)}}{\Pmat}}{\norm{\rr^{(0)}}{\Pmat}} \leq 10^{-6}.
\end{align*}

\subsection{Diagonal preconditioner and multilevel diagonal preconditioner}
For the wirebasket component we use a simple diagonal preconditioner. Indeed,
it is straightforward to prove that
\begin{align} \label{prec_wire_scale}
  \norm{v_{i,1}|_{\partial\Gamma_i}}{L^2(\partial\Gamma_i)}^2 \simeq \sum_{z_j\in \KK_i^1} \yy_j^2
  \norm{\eta_j^{(i)}}{L^2(\partial\Gamma_i)}^2
  \simeq \sum_{z_j\in \KK_i^1} \yy_j^2 \, \diam(\omega_j),
\end{align}
where $v_{i,1} = \sum_{z_j\in\KK_i^1} \yy_j \eta_j^{(i)}\in X_{i,1}$ and 
$\diam(\omega_j)$ is the diameter of the node patch $\omega_j = \supp(\eta_j^{(i)})$ of $z_j$.
For the example from Section~\ref{sec:examples:Zshape}, we define the diagonal preconditioner
\begin{align*}
  (\Pmat_{\partial\Gamma_i})_{jk} := \frac{|\omega_j|^{1/2}}{12} \, \delta_{jk}.
\end{align*}
According to \eqref{prec_wire_scale} it is optimal, that is, the numbers from~\eqref{eq:defPrecL2} behave like
\begin{align}\label{eq:equivConst:diag}
  \mu_\mathrm{min}^{(i)} \simeq \mu_\mathrm{max}^{(i)} \simeq 1,\quad i=1,\ldots, N.
\end{align}
For the example from Section~\ref{sec:examples:Quad} we test as preconditioner (for the wirebasket components)
the diagonal of the matrix, i.e., we set
\begin{align*}
  (\Pmat_{\partial\Gamma_i})_{jk} := \dual{\slp_i \curl_H \eta_k^{(i)}}{\curl_H\eta_j^{(i)}}_{\Gamma_i} \delta_{jk}.
\end{align*}
Since $\dual{\slp_i \curl_H \eta_j^{(i)}}{\curl_H\eta_j^{(i)}}_{\Gamma_i}\simeq \diam(\omega_j) \simeq
|\omega_j|^{1/2}$ the constants from~\eqref{eq:defPrecL2} satisfy \eqref{eq:equivConst:diag} in this case as well.

It remains to select preconditioners for the matrix blocks that belong to the interior unknowns, i.e.,
the ones corresponding to the nodes $\KK_i^0$, $i=1,\ldots, N$. As indicated by \eqref{eq:defPrecHyp}, it is enough
to take for each subdomain a standard preconditioner that works for the hypersingular operator.
In the following we use as $\Pmat_{\hyp_i}$ a multilevel diagonal preconditioner, i.e.,
\begin{align*} %\label{eq:def:mld}
  \Pmat_{\hyp_i}^{-1} := \sum_{\ell_i=0}^{L_i} \TransferMat_{\ell_i} \DiagMat_{\ell_i}^{-1} \TransferMat_{\ell_i}^T.
\end{align*}
More precisely, we consider $\TT_i=\TT_{i,L_i}$ as the finest level of a sequence of meshes
$\TT_{i,\ell}$ ($\ell = 0,\dots,L_i$). Then,
$\DiagMat_{\ell_i}$ is the diagonal part of the Galerkin matrix of
$\dual{\slp \curl_H (\cdot)}{\curl_H(\cdot)}_{\Gamma_i}$ with respect to the nodal basis of $X_{i,0}$ on level $\ell_i$ and
$\TransferMat_{\ell_i}$ is the matrix representation of the embedding operator which embeds elements of the space
$X_{i,0}$ on a coarse level $\ell_i$ to functions on the fine level $L_i$.
For the examples from Section~\ref{sec:examples:Zshape}, we replace the entries of $(\DiagMat_{\ell_i})_{jj}$ by
$|\omega_j^{\ell_i}|^{1/2}/12$. Here, $\omega_j^{\ell_i}$ is the support of the basis functions of level $\ell_i$
associated with node $j$.

It is known, see, e.g.,~\cite{amcl03}, that these preconditioners are optimal on triangular meshes, i.e., the constants
from~\eqref{eq:defPrecHyp} satisfy
\begin{align}\label{eq:equivConst:mld}
  \evmin^{(i)} \simeq \evmax^{(i)} \simeq 1
\end{align}
with mesh size independent constants.
Such multilevel preconditioners can be extended to locally refined meshes with assumptions on refinement zones, see,
e.g.,~\cite{amcl03}, or by use of special refinement strategies like Newest Vertex Bisection, cf.~\cite{ffps}.
The basic idea is that smoothing with the diagonal elements is done with respect to the degrees of freedom, where
the associated basis functions have changed.

We remark that the cited results for the multilevel diagonal preconditioners are stated for triangular meshes only.
However, for uniform refinements, the same techniques can be used to prove optimality on quadrilateral
meshes.
Finally, note that~\eqref{eq:equivConst:diag} and~\eqref{eq:equivConst:mld} imply that the numbers
$\evmin,\evmax$ from Theorem~\ref{thm:main} satisfy
\begin{align*}
  \evmin \simeq \evmax \simeq 1.
\end{align*}
According to Theorem~\ref{thm:cond} we then expect bounds
$\kappa(\widetilde\sysmat)=O(|\!\log\underline{h}|^\beta)$ with $\beta\le 5$ in \emph{Case 1a} and
$\beta\le 4$ in \emph{Cases 2,3}.
A theoretical bound for the condition number in \emph{Case 1b} would results in an exponent $\beta>5$
(and is not given here). But our numerical results show that this preconditioner is as competitive as in
\emph{Case 1a}, and better than in \emph{Cases 2} and \emph{3}. Our explanation is that some of the
technical bounds used in proofs are not sharp, see the discussion in the introduction.

\subsection{Problem on Z-shaped domain with triangular meshes}\label{sec:examples:Zshape}
We consider the variational formulation~\eqref{eq:mortarbem:alt} with $f=1$ and stabilization parameter $\alpha = 0.1$
on the Z-shaped domain $\Gamma$ from Figure~\ref{fig:Zshape}.
In this case we consider only one subdomain, i.e., $\TT=\{\Gamma\}$, $N=1$.
The stabilization parameter is chosen such that the lower order stabilization terms in the
definition~\eqref{eq:def:blfalt} are not the dominating parts in the condition numbers (for large $\hmax$).

\begin{figure}[htb]
  \begin{center}
    \psfrag{x}[c][c]{\tiny $x$}
    \psfrag{y}[c][c]{\tiny $y$}
    \includegraphics[width=0.6\textwidth]{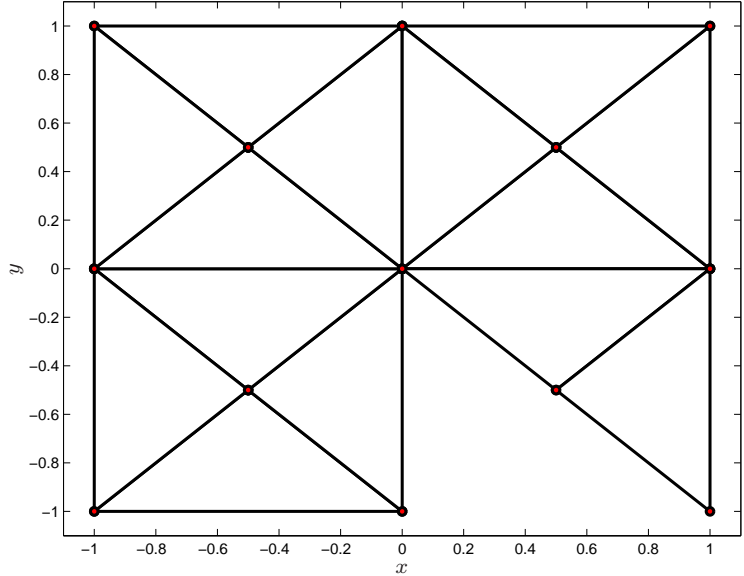}
    \caption{Initial triangulation of Z-shaped domain for the example from Section~\ref{sec:examples:Zshape}.}
    \label{fig:Zshape}
  \end{center}
\end{figure}

\begin{figure}[htb]
  \begin{center}
    \psfrag{dof}[c][c]{\tiny degrees of freedom}
    \psfrag{condition number}[c][c]{\tiny condition numbers}
    \psfrag{iterations}[c][c]{\tiny number of iterations}
    \psfrag{noPrec}{\tiny no prec.}
    \psfrag{diag}{\tiny diag.}
    \psfrag{case1a}{\tiny Case 1a}
    \psfrag{case1b}{\tiny Case 1b}
    \psfrag{case2}{\tiny Case 2}
    \psfrag{case3}{\tiny Case 3}
    \includegraphics[width=0.49\textwidth]{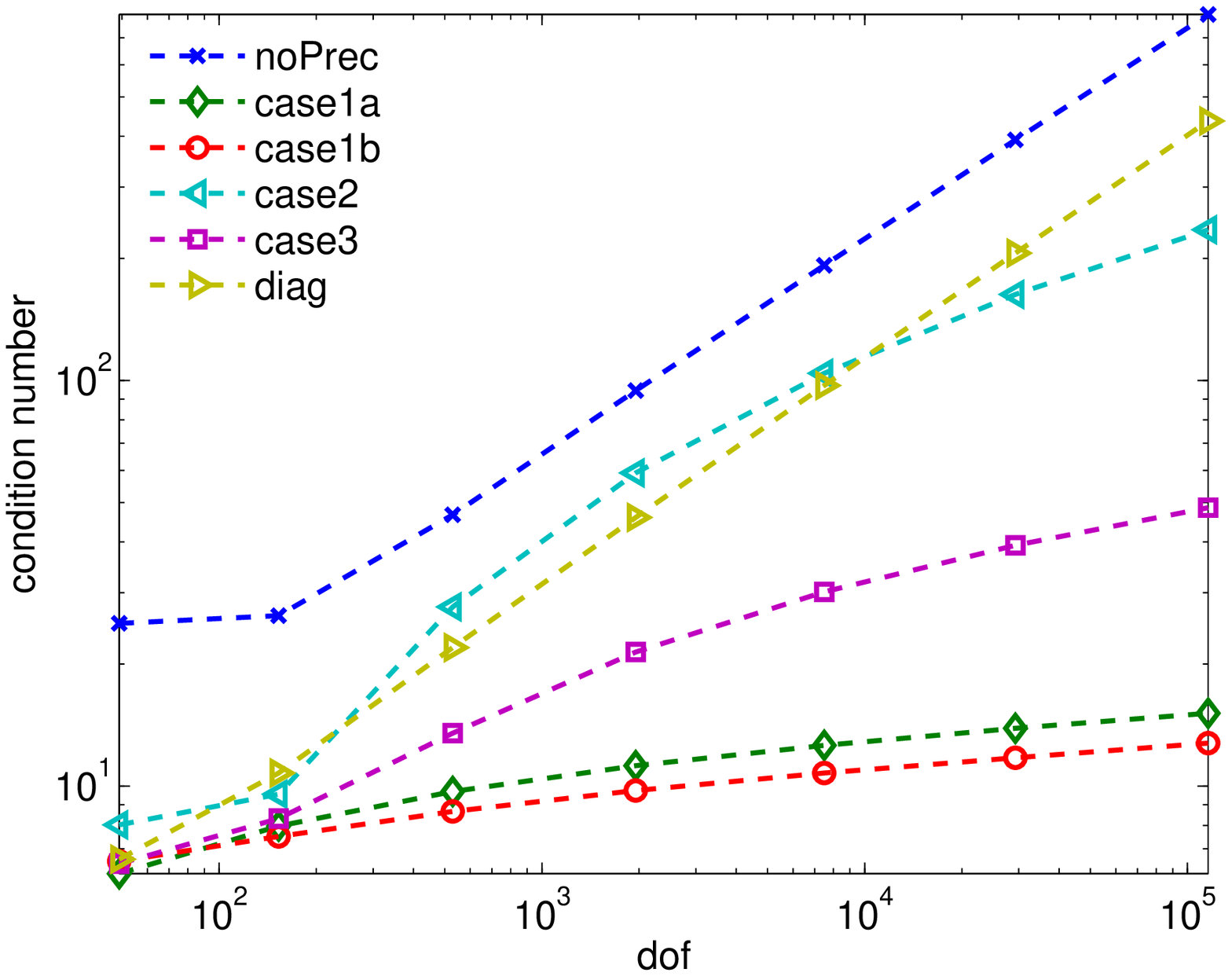}
    \includegraphics[width=0.49\textwidth]{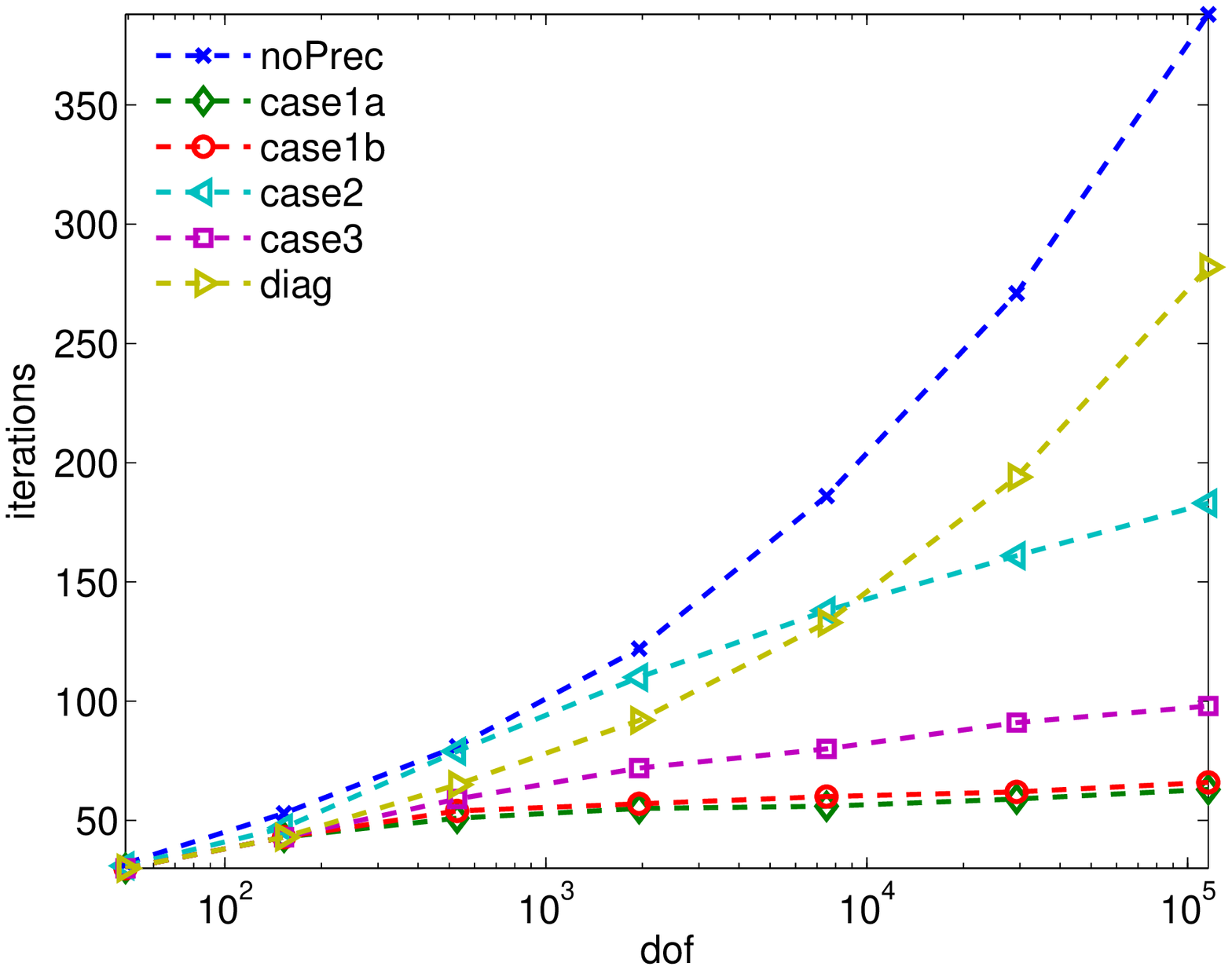}
    \caption{Condition numbers of the preconditioned systems and number of iterations in the MINRES algorithm for the
      example of Section~\ref{sec:examples:Zshape} and uniform refinements.}
    \label{fig:Zshape:unif}
  \end{center}
\end{figure}

\begin{table}
  \begin{center}
\begin{tabular}{|c|c|c|c|c|c|c|c|c|c|}
\hline
\textbf{step}&\textbf{dof}&\textbf{$\hmax$}&\textbf{$\hmin$}&\textbf{no prec.}&\textbf{1a}&\textbf{1b}&\textbf{2}&\textbf{3}&\textbf{diag.}\\
\hline\hline
1 & 49 & 1/2 & 1/2 & 25.19 & 6.09 & 6.52 & 8.03 & 6.44 & 6.60\\\hline
2 & 153 & 1/4 & 1/4 & 26.33 & 7.97 & 7.51 & 9.56 & 8.31 & 10.77\\\hline
3 & 529 & 1/8 & 1/8 & 46.68 & 9.70 & 8.66 & 27.66 & 13.50 & 21.95\\\hline
4 & 1953 & 1/16 & 1/16 & 94.46 & 11.22 & 9.76 & 59.18 & 21.41 & 45.98\\\hline
5 & 7489 & 1/32 & 1/32 & 192.36 & 12.60 & 10.77 & 104.18 & 30.11 & 97.20\\\hline
6 & 29313 & 1/64 & 1/64 & 392.15 & 13.88 & 11.75 & 162.97 & 39.24 & 206.00\\\hline
7 & 115969 & 1/128 & 1/128 & 799.48 & 15.12 & 12.77 & 235.40 & 48.56 & 436.41\\\hline
\end{tabular}

    \vspace{1ex}
  \end{center}
  \caption{Condition numbers of the preconditioned systems for the
      example of Section~\ref{sec:examples:Zshape} and uniform refinements.}
  \label{tab:Zshape:unif}
\end{table}

\begin{figure}[htb]
  \begin{center}
    \psfrag{dof}[c][c]{\tiny degrees of freedom}
    \psfrag{condition number}[c][c]{\tiny condition numbers}
    \psfrag{case2}{\tiny Case 2}
    \psfrag{case3}{\tiny Case 3}
    \psfrag{log3}{\tiny $|\!\log(\underline{h})|^3$}
    \psfrag{log2}{\tiny $|\!\log(\underline{h})|^2$}
    \includegraphics[width=0.7\textwidth,height=0.5\textwidth]{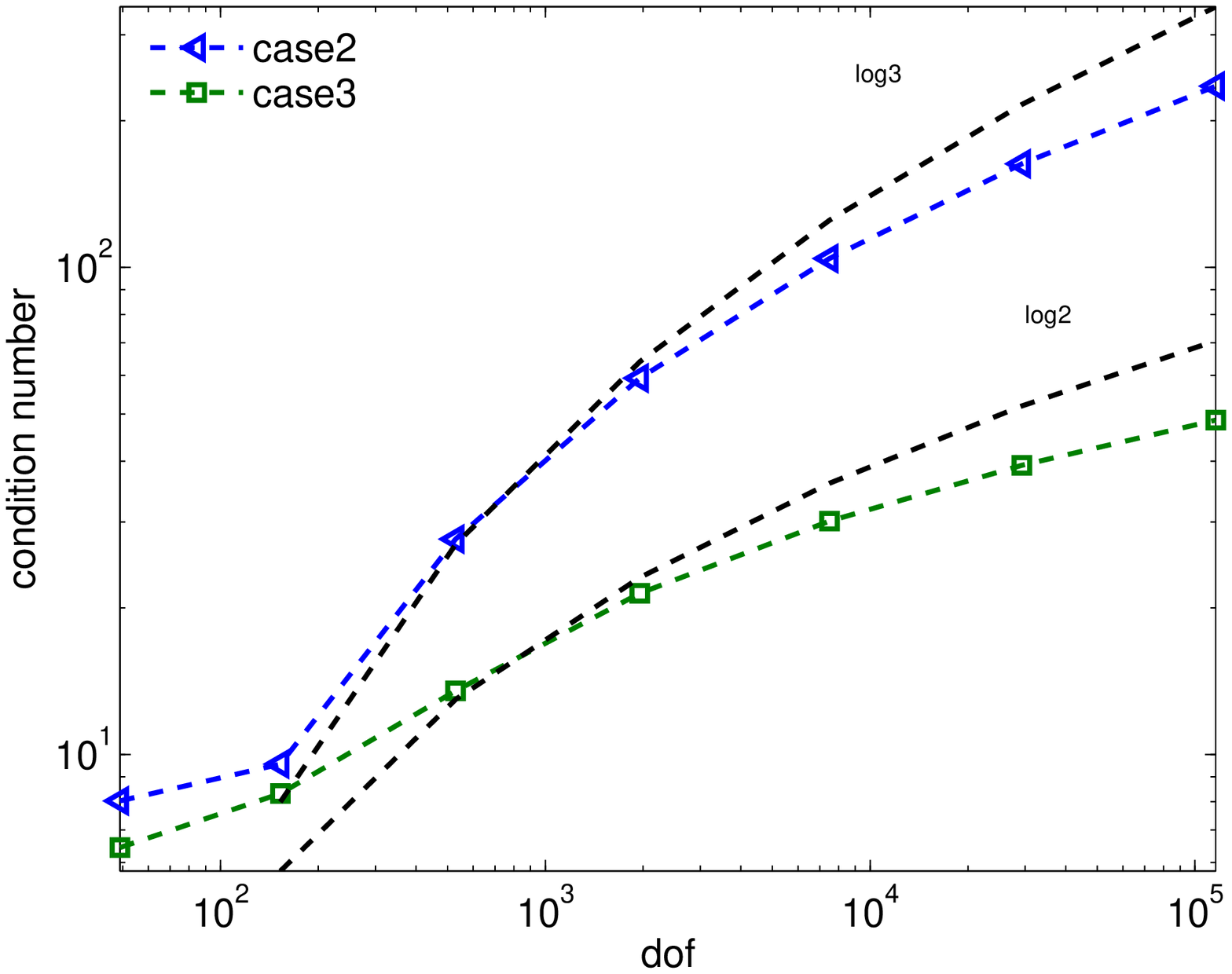}
    \caption{Logarithmic behavior of condition numbers for
      example of Section~\ref{sec:examples:Zshape} and uniform refinements.}
    \label{fig:Zshape:unif_log}
  \end{center}
\end{figure}

For the definition of the Lagrangian multiplier space $Y_h$, we combine two adjacent boundary edges to one element of
the mesh $\tau_\ell$.
Mesh refinement is driven by Newest Vertex Bisection, see, e.g.,~\cite{kpp}. In particular, we note
that each triangle $T$ is divided into 4 son elements $T_1,\dots,T_4$, with $|T_j|=|T|/4$. Moreover, this refinement
rule preserves shape-regularity, i.e.
\begin{align*}
  \sup_{T\in\TT_\ell} \frac{\diam(T)^2}{|T|} \lesssim  \sup_{T\in\TT_0} \frac{\diam(T)^2}{|T|},
\end{align*}
which also holds for adaptive mesh refinements.
For details we refer the interested reader to~\cite{kpp} and references therein.
We remark that the initial triangulation does not satisfy Assumption~\ref{ass:interfacemesh} since, for instance, the
boundary edge $(-1,0)\times\{0\}\times\{0\}$ of $\TT_0$ contains only one boundary element. We use a uniform refinement
in the first step, which ensures that Assumption~\ref{ass:interfacemesh} holds true.

Figure~\ref{fig:Zshape:unif} shows the condition numbers (left) as well as the numbers of iterations
(right) needed to reduce the relative residual in the MINRES method by $10^{-6}$ in the case of uniform refinements.
Additionally, the condition numbers are listed in Table~\ref{tab:Zshape:unif}.

In the following let us refer to $P_j$ as the preconditioner of ``\emph{Case $j$}'' ($j\in\{1a,1b,2,3\}$).
The numerical results indicate that the preconditioners $P_{1a}$ and $P_{1b}$ are better than the others,
and that $P_3$ is better than $P_2$. In contrast, Theorem~\ref{thm:cond}
predicts better bounds for $P_2$ and $P_3$. Nevertheless, all the results confirm the theoretical estimates.
Indeed, Figure~\ref{fig:Zshape:unif_log} indicates that $\kappa(\widetilde\sysmat)$ is bounded by
$O(|\!\log(\underline{h})|^3)$ even in \emph{Case 2}.
%This indicates numerically that the estimates from Remark~\ref{rem:logbounds} hold true.

\begin{figure}[htb]
  \begin{center}
    \psfrag{dof}[c][c]{\tiny degrees of freedom}
    \psfrag{condition number}[c][c]{\tiny condition numbers}
    \psfrag{iterations}[c][c]{\tiny number of iterations}
    \psfrag{noPrec}{\tiny no prec.}
    \psfrag{diag}{\tiny diag.}
    \psfrag{case1a}{\tiny Case 1a}
    \psfrag{case1b}{\tiny Case 1b}
    \psfrag{case2}{\tiny Case 2}
    \psfrag{case3}{\tiny Case 3}
    \includegraphics[width=0.49\textwidth]{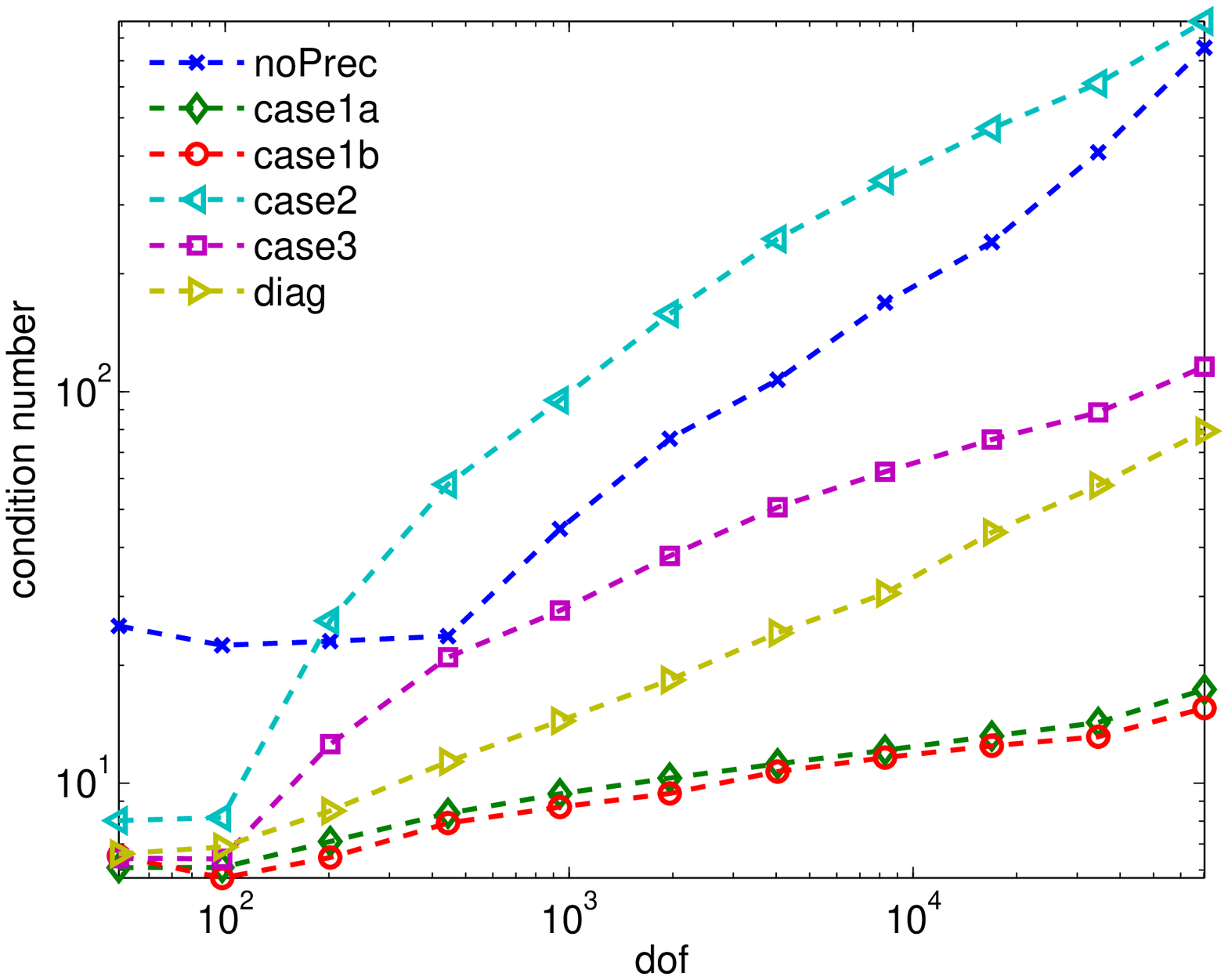}
    \includegraphics[width=0.49\textwidth]{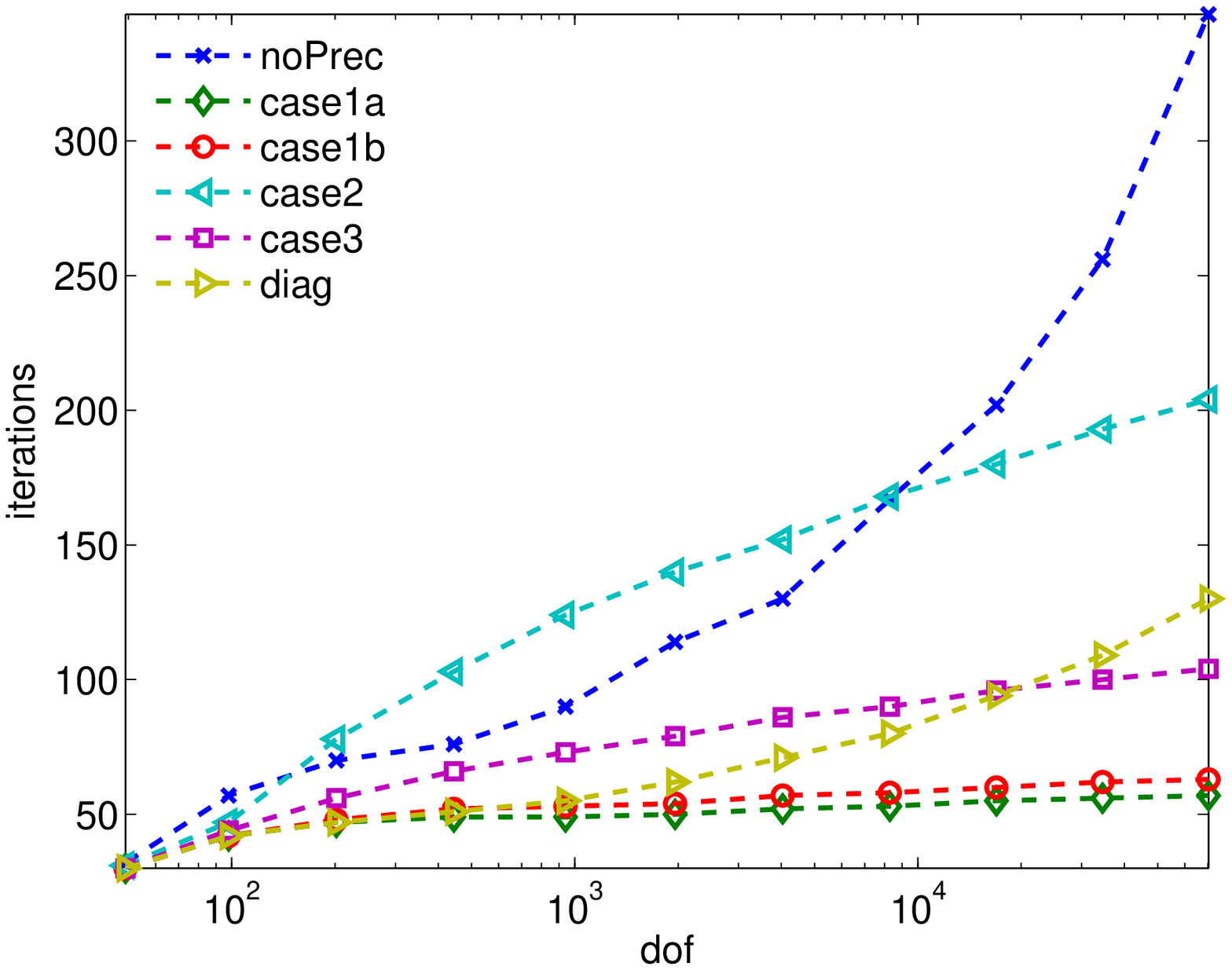}
    \caption{Condition numbers of the preconditioned systems and number of iterations in the MINRES algorithm for the
      example of Section~\ref{sec:examples:Zshape} and adaptive refinements.}
    \label{fig:Zshape:adap}
  \end{center}
\end{figure}

\begin{table}
  \begin{center}
\begin{tabular}{|c|c|c|c|c|c|c|c|c|c|}
\hline\textbf{step}&\textbf{dof}&\textbf{$\hmax$}&\textbf{$\hmin$}&\textbf{no prec.}&\textbf{1a}&\textbf{1b}&\textbf{2}&\textbf{3}&\textbf{diag.}\\
\hline\hline
1 & 49 & 5.00e-01 & 1/2 & 25.19 & 6.09 & 6.52 & 8.03 & 6.44 & 6.60\\\hline
2 & 98 & 5.00e-01 & 1/4 & 22.51 & 6.09 & 5.73 & 8.17 & 6.40 & 6.88\\\hline
3 & 202 & 5.00e-01 & 1/8 & 23.05 & 7.09 & 6.46 & 26.00 & 12.58 & 8.49\\\hline
4 & 444 & 5.00e-01 & 1/16 & 23.74 & 8.37 & 7.91 & 57.98 & 20.98 & 11.35\\\hline
5 & 939 & 5.00e-01 & 1/32 & 44.60 & 9.40 & 8.69 & 95.06 & 27.59 & 14.42\\\hline
6 & 1961 & 3.54e-01 & 1/64 & 75.78 & 10.31 & 9.42 & 158.09 & 38.06 & 18.36\\\hline
7 & 4038 & 2.50e-01 & 1/128 & 107.28 & 11.21 & 10.71 & 245.60 & 50.67 & 24.21\\\hline
8 & 8289 & 1.77e-01 & 1/256 & 168.65 & 12.13 & 11.62 & 345.67 & 62.42 & 30.52\\\hline
9 & 16939 & 1.25e-01 & 1/512 & 240.84 & 13.20 & 12.46 & 470.05 & 75.43 & 43.72\\\hline
10 & 34516 & 1.25e-01 & 1/1024 & 408.58 & 14.28 & 13.15 & 612.52 & 88.46 & 57.62\\\hline
11 & 70278 & 8.84e-02 & 1/2048 & 755.93 & 17.34 & 15.54 & 882.78 & 115.87 & 79.32\\\hline
\end{tabular}

    \vspace{1ex}
  \end{center}
  \caption{Condition numbers of the preconditioned systems for the
      example of Section~\ref{sec:examples:Zshape} and adaptive refinements.}
  \label{tab:Zshape:adap}
\end{table}

\begin{figure}[htb]
  \begin{center}
    \psfrag{dof}[c][c]{\tiny degrees of freedom}
    \psfrag{condition number}[c][c]{\tiny condition numbers}
    \psfrag{case2}{\tiny Case 2}
    \psfrag{case3}{\tiny Case 3}
    \psfrag{log3}{\tiny $|\!\log(\underline{h})|^3$}
    \psfrag{log2}{\tiny $|\!\log(\underline{h})|^2$}
    \includegraphics[width=0.7\textwidth,height=0.5\textwidth]{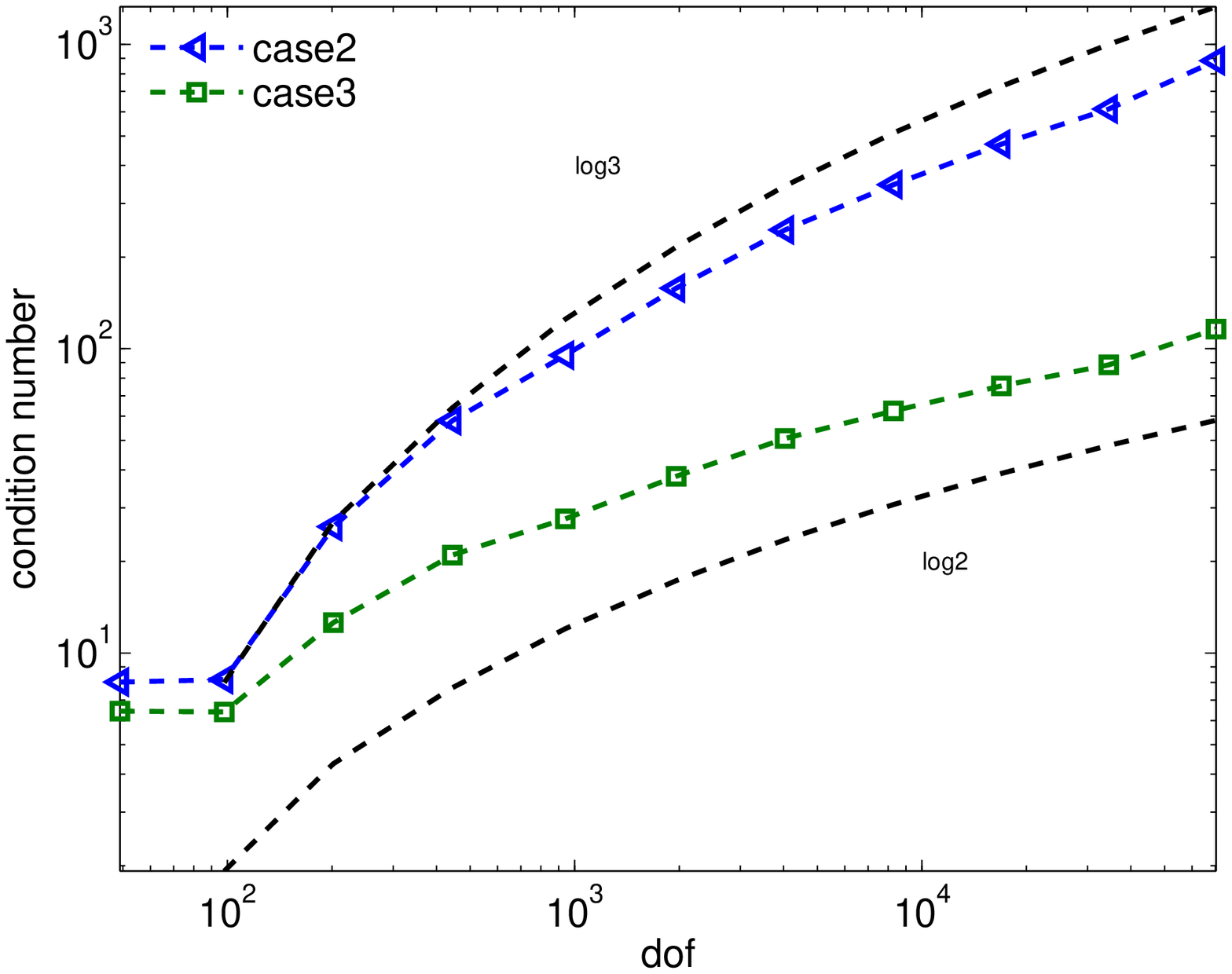}
    \caption{Logarithmic behavior of condition numbers for
      example of Section~\ref{sec:examples:Zshape} and adaptive refinements.}
    \label{fig:Zshape:Adap_log}
  \end{center}
\end{figure}

In the next example we consider adaptive mesh refinements, where we use a simple ZZ-type estimator,
see, e.g.,~\cite{arcme}, to mark elements for refinement and additionally refine all elements that share a boundary edge.
We note that this estimator is not analyzed in~\cite{arcme} for the present (non-conforming) situation,
but is heuristically used to obtain adaptively refined meshes.
Condition numbers of the preconditioned systems and numbers of iterations needed in the MINRES algorithm
are plotted in Figure~\ref{fig:Zshape:adap}. Moreover, the condition numbers are listed in Table~\ref{tab:Zshape:adap}.
We observe similar results as in the case of uniform refinements.
In particular, our theoretical results are confirmed also for adaptively refined meshes.
Again, the results for the weakest of the domain decomposition preconditioners, $P_2$, indicate that
$\kappa(\widetilde\sysmat)\lesssim O(|\!\log(\underline{h})|^3)$ also in this case,
cf.~Figure~\ref{fig:Zshape:Adap_log}.

\subsection{Problem with four subdomains and quadrilateral meshes}\label{sec:examples:Quad}
We consider the variational formulation~\eqref{eq:mortarbem:alt} with $f=1$ and stabilization parameter $\alpha = 0.1$
on the quadratic domain $\Gamma := (0,2)^2 \times \{0\}$ with a decomposition into four subdomains $\Gamma_1 = (0,1)^2
\times \{0\}$, $\Gamma_2 = (1,2)\times(0,1)\times \{0\}$, $\Gamma_3 = (0,1)\times(1,2)\times \{0\}$, and $\Gamma_4 =
(1,2)^2\times \{0\}$ sketched in Figure~\ref{fig:Quad}.
For the intersections $\overline\Gamma_1\cap\overline\Gamma_2$ and $\overline\Gamma_1\cap\overline\Gamma_3$, 
we define $\Gamma_1$ to be the Lagrangian side and
for the intersections $\overline\Gamma_4\cap\overline\Gamma_2$ and $\overline\Gamma_4\cap\overline\Gamma_3$, 
we define $\Gamma_4$ to be the Lagrangian side.
We define the Lagrangian elements that come from $\Gamma_1$ and $\Gamma_4$ as the union of two adjacent edges that lie
in $\partial\Gamma_i$.
For the Lagrangian elements that come from $\Gamma_2$ and $\Gamma_3$ we take the union of three adjacent boundary edges.

\begin{figure}[htb]
  \begin{center}
    \psfrag{x}[c][c]{\tiny $x$}
    \psfrag{y}[c][c]{\tiny $y$}
    \psfrag{G1}[c][c]{$\Gamma_1$}
    \psfrag{G2}[c][c]{$\Gamma_2$}
    \psfrag{G3}[c][c]{$\Gamma_3$}
    \psfrag{G4}[c][c]{$\Gamma_4$}
    \includegraphics[width=0.6\textwidth]{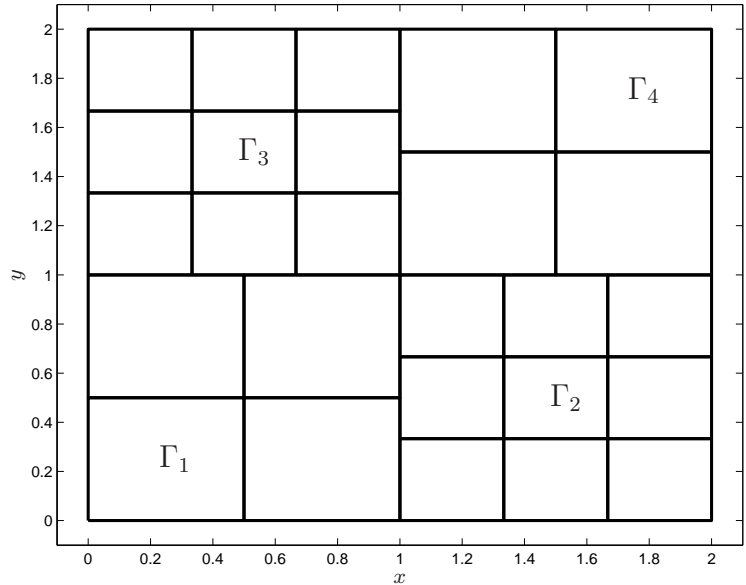}
    \caption{Subspace decomposition of $\Gamma = (0,2)^2 \times\{0\}$ and their initial meshes for the example from
    Section~\ref{sec:examples:Quad}.}
    \label{fig:Quad}
  \end{center}
\end{figure}

\begin{figure}[htb]
  \begin{center}
    \psfrag{dof}[c][c]{\tiny degrees of freedom}
    \psfrag{condition number}[c][c]{\tiny condition numbers}
    \psfrag{iterations}[c][c]{\tiny number of iterations}
    \psfrag{noPrec}{\tiny no prec.}
    \psfrag{diag}{\tiny diag.}
    \psfrag{case1a}{\tiny Case 1a}
    \psfrag{case1b}{\tiny Case 1b}
    \psfrag{case2}{\tiny Case 2}
    \psfrag{case3}{\tiny Case 3}
    \includegraphics[width=0.49\textwidth]{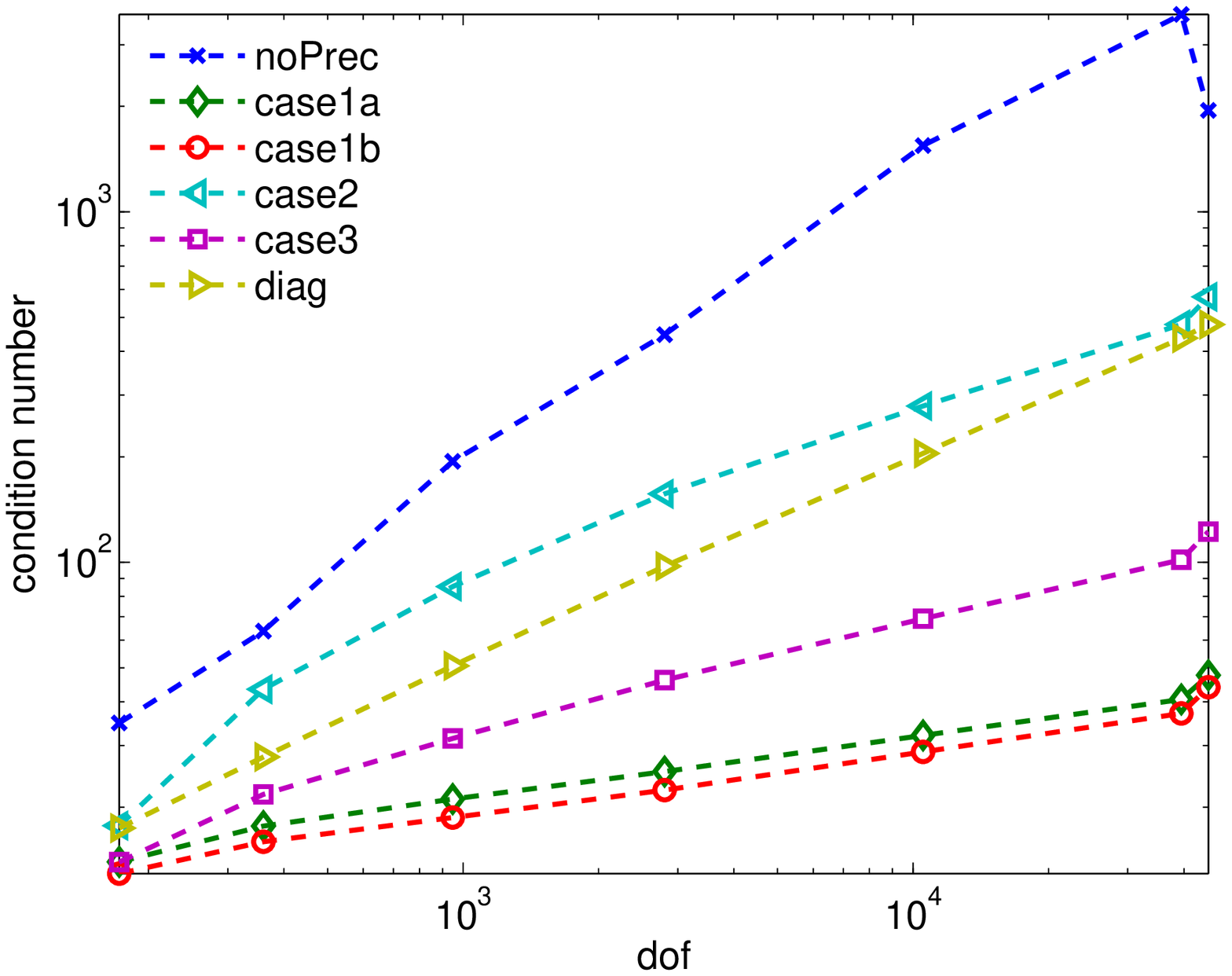}
    \includegraphics[width=0.49\textwidth]{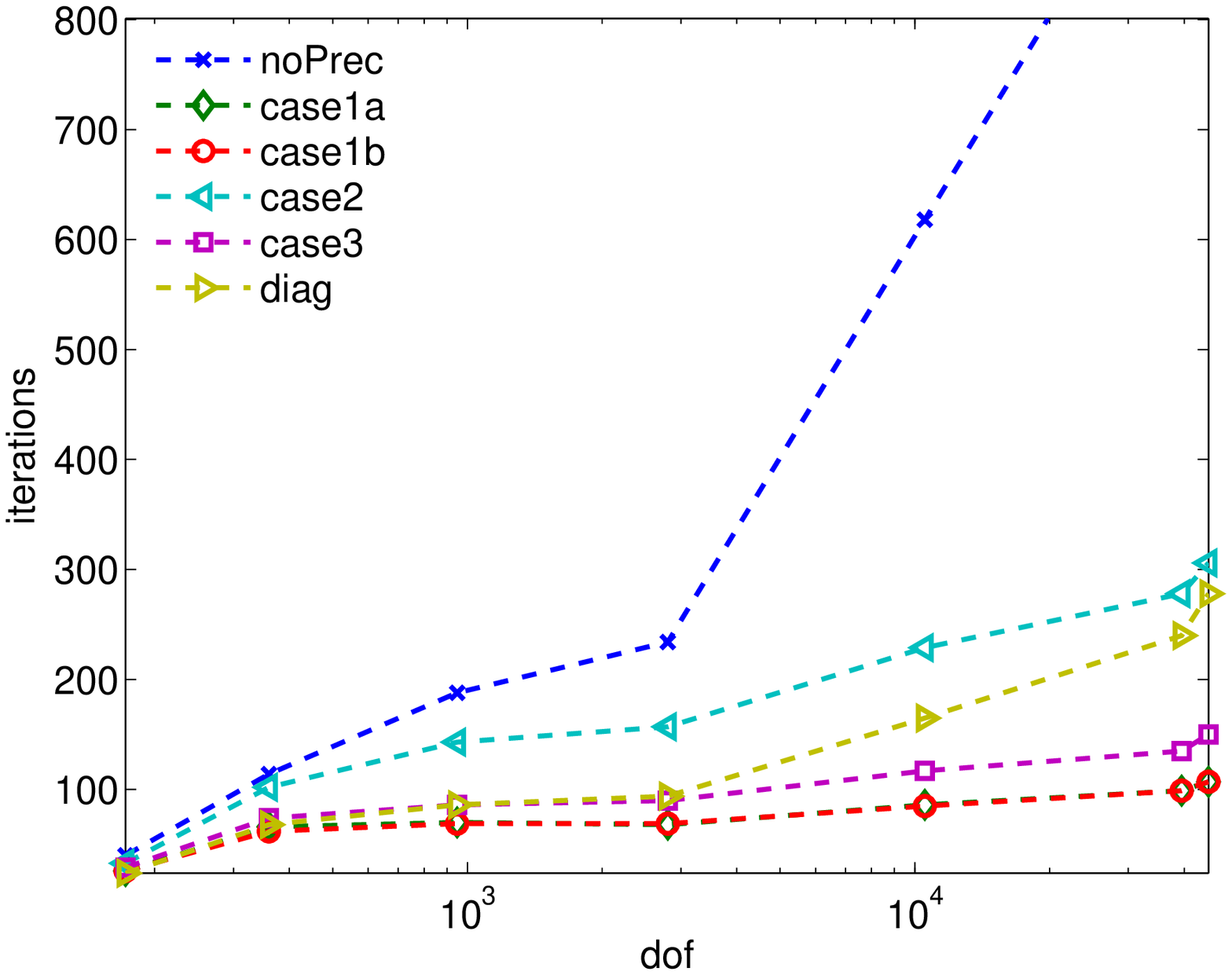}
    \caption{Condition numbers of the preconditioned systems and number of iterations in the MINRES algorithm for the
      example of Section~\ref{sec:examples:Quad} with different refinement levels of subdomain meshes.}
    \label{fig:Quad:unif}
  \end{center}
\end{figure}

\begin{table}
  \begin{center}
\begin{tabular}{|c|c|c|c|c|c|c|c|c|c|c|c|}
\hline\textbf{step}&\textbf{dof}&\textbf{$\hmin_{1}$}&\textbf{$\hmin_{2}$}&\textbf{$\hmin_{3}$}&\textbf{$\hmin_{4}$}&\textbf{no prec.}&\textbf{1a}&\textbf{1b}&\textbf{2}&\textbf{3}&\textbf{diag.}\\
\hline\hline
1 & 172 & 1/4 & 1/6 & 1/6 & 1/4 & 34.77 & 13.97 & 12.94 & 17.72 & 13.96 & 17.43\\\hline
2 & 360 & 1/4 & 1/12 & 1/6 & 1/8 & 63.62 & 17.67 & 15.95 & 43.41 & 21.77 & 27.83\\\hline
3 & 948 & 1/4 & 1/24 & 1/12 & 1/8 & 194.24 & 21.10 & 18.70 & 85.29 & 31.43 & 50.64\\\hline
4 & 2804 & 1/8 & 1/48 & 1/12 & 1/8 & 446.57 & 25.27 & 22.39 & 156.85 & 46.10 & 97.47\\\hline
5 & 10532 & 1/8 & 1/96 & 1/24 & 1/16 & 1545.48 & 32.08 & 28.76 & 279.10 & 69.07 & 204.58\\\hline
6 & 39492 & 1/16 & 1/192 & 1/24 & 1/32 & 3661.95 & 40.65 & 37.04 & 476.84 & 101.74 & 436.13\\\hline
7 & 45316 & 1/32 & 1/192 & 1/48 & 1/64 & 1948.27 & 47.65 & 43.99 & 572.20 & 122.29 & 477.40\\\hline
\end{tabular}

    \vspace{1ex}
  \end{center}
  \caption{Condition numbers of the preconditioned systems for the
    example of Section~\ref{sec:examples:Quad} with different refinement levels of subdomain meshes.}
  \label{tab:Quad:unif}
\end{table}

\begin{figure}[htb]
  \begin{center}
    \psfrag{dof}[c][c]{\tiny degrees of freedom}
    \psfrag{condition number}[c][c]{\tiny condition numbers}
    \psfrag{case2}{\tiny Case 2}
    \psfrag{case3}{\tiny Case 3}
    \psfrag{log3}{\tiny $|\!\log(\underline{h})|^3$}
    \psfrag{log2}{\tiny $|\!\log(\underline{h})|^2$}
    \includegraphics[width=0.7\textwidth,height=0.5\textwidth]{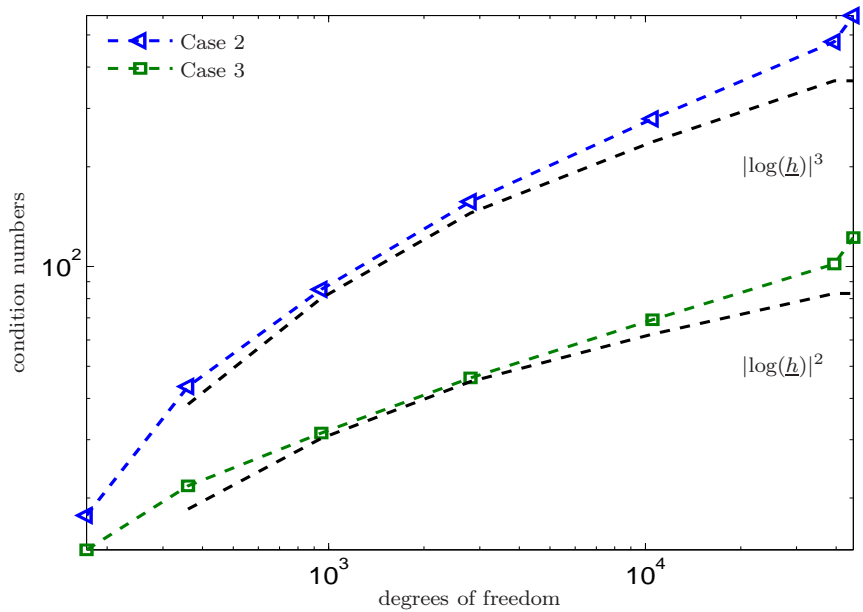}
    \caption{Logarithmic behavior of condition numbers for
      example of Section~\ref{sec:examples:Quad} with different refinement levels of subdomain meshes.}
    \label{fig:Zshape:Quad_log}
  \end{center}
\end{figure}

We consider uniform refinements where each element of $\TT_j$ is divided into four elements.
For the experiment we refine each of the subdomain meshes separately, which leads to different mesh sizes
$\hmin_1$, $\hmin_2$, $\hmin_3$, $\hmin_4$. Note that for our problem configuration there holds $\hmax_j=\hmin_j$.
The results are given in Figure~\ref{fig:Quad:unif} and Table~\ref{tab:Quad:unif}.
As in Section~\ref{sec:examples:Zshape} we observe that the preconditioners $P_{1a}$, $P_{1b}$ corresponding
to the preconditioning form $d_1(\cdot,\cdot)$ behave best in terms of condition numbers and numbers of iterations.
The preconditioners $P_2$ and $P_3$ stemming, respectively, from $d_2(\cdot,\cdot)$ and $d_3(\cdot,\cdot)$
show a stronger dependence on the mesh size. Nevertheless, theoretical bounds are confirmed also
for this example. In particular, Figure~\ref{fig:Zshape:Quad_log} suggests that
$\kappa(\widetilde\sysmat)\lesssim O(|\!\log(\underline{h})|^3)$ for all domain decomposition preconditioners.

Let us also remark that the condition number of the un-preconditioned system gets smaller from
Step~6 to Step~7, see Table~\ref{tab:Quad:unif}.
The condition number $\kappa(\sysmat)$ is bounded (up to logarithmic terms) by $\evmax^{1/2}/\evmin$ and this term
depends on the ratio $\hmax^{1/2}/\hmin$. Since $\hmax$ gets smaller and $\hmin$ stays constant from Step~6 to Step~7 (as we
refine all subdomains except $\Gamma_2$), this explains the observation.

%%%%%%%%%%%%%%%%%%%%%%%%%%%%%%%%%%%%%%%%%%%%%%%%%%%%%%%%%%%%%%%%%%%%%
% Bibliography
\bibliographystyle{plain}
\bibliography{literature}
%%%%%%%%%%%%%%%%%%%%%%%%%%%%%%%%%%%%%%%%%%%%%%%%%%%%%%%%%%%%%%%%%%%%%

\end{document}